\documentclass[12pt]{amsart}
\usepackage{amsmath, amsthm, amscd, amssymb, amsfonts, latexsym}
\usepackage{fullpage}
\usepackage{bm}
\usepackage[all]{xy}
\usepackage{tikz}
\usepackage{mathrsfs}
\usepackage{ wasysym }
\usepackage{dsfont}
\usepackage[usestackEOL]{stackengine}
\stackMath

\newcommand{\kommentar}[1]{}

\newcommand{\acom}[1]{{\color{blue}{Alexandra: #1}} }

\newcommand{\F}{\mathbb F}

\newcommand{\as}{\mathcal{AS}_d^0}

\newcommand{\Z}{\mathbb Z}

\newcommand{\R}{\mathbb R}

\newcommand{\C}{\mathbb C}

\DeclareMathOperator{\tr}{tr}

\DeclareMathOperator{\re}{Re}

\renewcommand{\pmod}[1]{\,(\mathrm{mod}\,#1)}

\newtheorem{lem}{Lemma}[section]
\newtheorem{prop}[lem]{Proposition}
\newtheorem{thm}[lem]{Theorem}

\newtheorem{conj}[lem]{Conjecture}
\newtheorem{rmk}[lem]{Remark}
\theoremstyle{definition}

\begin{document}

\title{Moments of Artin--Schreier $L$-functions}
\author{Alexandra Florea, Edna Jones, Matilde Lalin}

\address{Alexandra Florea: Department of Mathematics, UC Irvine, 340 Rowland Hall, Office 540E, Irvine, CA 92697, USA}\email{floreaa@uci.edu}

\address{Edna Jones: Department of Mathematics,
Duke University,
120 Science Drive,
Durham, NC 27708, USA} \email{edna.jones@math.duke.edu}

\address{Matilde Lal\'in: 
D\'epartement de math\'ematiques et de statistique,
                                    Universit\'e de Montr\'eal,
                                    CP 6128, succ. Centre-ville,
                                     Montreal, QC H3C 3J7, Canada}\email{matilde.lalin@umontreal.ca}

\begin{abstract}
We compute moments of $L$-functions associated to  the polynomial family of Artin--Schreier covers over $\mathbb{F}_q$, where $q$ is a power of a prime $p>2$, when the size of the finite field is fixed and the genus of the family goes to infinity. More specifically, we compute the $k^{\text{th}}$ moment for a large range of values of $k$, depending on the sizes of $p$ and $q$. We also compute the second moment in absolute value of the polynomial family, obtaining an exact formula with a lower order term, and confirming the unitary symmetry type of the family. 
\end{abstract}


\maketitle

\section{Introduction}

In this paper, we are interested in evaluating moments of $L$-functions associated to Artin--Schreier covers of $\mathbb{P}^1$. Computing moments in families of $L$-functions has a long history. For example, the moments of the Riemann zeta-function $\zeta(s)$ were introduced by Hardy and Littlewood \cite{hl}, who obtained asymptotic formulas for the second moment. The fourth moment was studied in \cite{ingham, heath_brown4}. There has been a wealth of literature on moments in various other families of $L$-functions; for a (non-exhaustive) list, see for example \cite{Sound, young4, soundyoung, conreyiwaniec}.

Here, we focus on the moments of $L$-functions of Artin--Schreier curves. These form an interesting family with a rich arithmetic structure. Their zeta functions are expressed in terms of additive characters of $\mathbb{F}_p$, not in terms of multiplicative characters (as in the case of hyperelliptic curves or cyclic $\ell$-covers, for example). The terms corresponding to a fixed additive character can be expressed as exponential sums. The extra arithmetic structure can be used to refine the Weil bound on Artin--Schreier curves \cite{RLW}.

Statistics of zeros of Artin--Schreier $L$-functions have been extensively studied. When the size of the finite field goes to infinity, one can use deep equidistribution results of Katz \cite{katz_monodromy}, building on work of Katz--Sarnak \cite{KS}, to show that the \textit{local} statistics are given by the corresponding statistics of eigenvalues of random matrices in certain ensembles, depending on the specific family under consideration. When considering Artin--Schreier curves, the $p$-rank introduces a stratification of the moduli space of covers of genus $g$ \cite{pz}. For example, $p$-rank $0$ corresponds to the family of \textit{polynomial Artin--Schreier curves}, while, when $(p-1)$ divides the genus, the maximal $p$-rank corresponds to the family of \textit{ordinary  Artin--Schreier curves}. Using the Katz--Sarnak results, one can show that in the large finite field limit, the local statistics in the polynomial family follow the local statistics of the unitary group of random matrices  (\cite[Theorem 3.9.2]{katz_monodromy}).

One can also consider the same statistics in the regime when the base finite field is fixed, and the genus of the family goes to infinity, in which case one cannot make use of the equidistribution results.  In this case,  as in the number field setting, one can usually compute only a few small moments (see, for example, \cite{Florea1, Florea2, Florea3, tamam, Djankovic}); a notable exception is the very recent work on moments of quadratic Dirichlet $L$--functions over function fields, which recovers all the moments \cite{hyper1, hyper2}.)

Entin \cite{Ent} considered the local statistics for the polynomial Artin--Schreier family and showed agreement with the random matrix model; these results were further improved and extended to the ordinary and odd polynomial families in recent work of Entin and Pirani \cite{EntPir}. The \textit{mesoscopic} statistics for the ordinary and polynomial families (as well as other $p$-rank strata) were considered in \cite{BDFLS12, BDFL16}, where the authors showed that the number of zeros with angles in a prescribed subinterval $I$ of $[-\pi,\pi]$ whose length is either fixed or goes to $0$, while $g|I| \to \infty$ (where $g$ denotes the genus of the family),  has a standard Gaussian distribution. One notices that the mesoscopic scale does not distinguish between the various Artin--Schreier families; hence, the local scale is a finer detector of the family structure. We note that the local statistics of zeros in the fixed finite field limit have been studied over function fields in the case of hyperelliptic curves \cite{Rudnick-traces, RG-traces, ERGR, Bui-Florea, BJ-traces}, cyclic $\ell$--covers \cite{BDFL, BDFL11, BDFKLOM}, non-cyclic cubic covers of $\mathbb{P}_{\mathbb{F}_q}^1$ \cite{BCDGLD, meisner2} and Dirichlet $L$-functions \cite{AMPT-zeros}. The distribution of zeros in the global and mesoscopic regimes were considered in \cite{FR} for hyperelliptic curves, in \cite{xiong1} for cyclic $\ell$-covers and in \cite{xiong2} for abelian covers of algebraic curves.

In this work we compute moments in the family of polynomial Artin--Schreier $L$-functions, and show that the moments for the polynomial family behave like the moments of the characteristic polynomials of random matrices in the unitary group. Moreover, we check that our answers agree with conjectures about moments \cite{CFKRS, Keating-Snaith}. Our results further support the Katz--Sarnak philosophy and agree with the behavior observed in Entin \cite{Ent} and Entin--Pirani \cite{EntPir} regarding the local statistics of zeros.

To describe our results, we first introduce some notation. Let $p>2$ be an odd prime, and $q$ a power of $p$. An Artin--Schreier curve is given by the affine equation
$$C_f: y^p-y = f(x),$$
where $f(x) \in \mathbb{F}_q(x)$ is a rational non-constant function, together with the automorphism $y \mapsto y+1$. Let $p_1,\ldots, p_{r+1}$ be the set of poles of $f(x)$ and let $d_j$ be the order of the pole $p_j$. Then the genus of $C_f$ is given by
$$\mathfrak{g}(C_f) = \frac{p-1}{2} \Big(-2+\sum_{j=1}^{r+1} (d_j+1) \Big) = \frac{p-1}{2} \Big( r-1+\sum_{j=1}^{r+1} d_j \Big).$$ 
To an Artin--Schreier curve one also associates its $p$--rank, which is defined to be the $\mathbb{Z}/p$-rank of $\text{Jac}(C_f \times \overline{\mathbb{F}_q})[p]$ (see, for example, \cite{pz}). A curve with $p$-rank $0$ is in the polynomial family, which corresponds to the case in which $f(x)$ is a polynomial. If we impose the extra condition that $f(-x) = -f(x)$, then the curve is in the odd polynomial family. When $p-1$ divides $\mathfrak{g}(C_f)$, a curve with $p$-rank equal to $\mathfrak{g}(C_f)$ is in the ordinary family.
We note that the techniques used to deal with the various subfamilies of Artin--Schreier $L$-functions are different and do not transfer from one subfamily to the others. In this paper, we focus on the family of polynomial Artin--Schreier $L$-functions. From now on, we assume that $f$ is a polynomial.

The zeta function of $C_f$ is given by 
$$Z_{C_f}(u) = \exp \Big( \sum_{k=1}^{\infty} N_k(C_f) \frac{u^k}{k} \Big),$$ where $N_k(C_f)$ denotes the number of points on $C_f$ over $\mathbb{F}_{q^k}$ (see, for example, \cite{Moreno, Rosen}). By the Weil conjectures, it follows that
\begin{equation} \label{def_l}
Z_{C_f}(u) = \frac{\mathcal{L}(u,C_f)}{(1-u)(1-qu)},
\end{equation}
where 
$\mathcal{L}(u,C_f)$ is the $L$-function associated to $C_f$, which is a polynomial of degree $2 \mathfrak{g}(C_f)$. It further follows that
$$\mathcal{L}(u,C_f) = \prod_{\psi \neq 1} \mathcal{L}(u,f,\psi),$$ 
where $\psi$ varies over the non-trivial additive characters of $\mathbb{F}_p$ and where 
\begin{equation*}
\mathcal{L}(u,f,\psi) =  \exp\Big(\sum_{n=1}^\infty S_n(f,\psi) \frac{q^{-ns}}{n} \Big) , \end{equation*}
 with
\begin{equation*}
S_n(f,\psi)=\sum_{\alpha \in \F_{q^n}} \psi(\tr_{q^n/p}(f(\alpha))). 
\end{equation*}
(See, for example, \cite{BDFLS12,BDFL16}.)

In the above,  $\tr_{q^n/p} : \mathbb{F}_{q^n} \to \mathbb{F}_p$ denotes the absolute trace map.

Here, we will consider the family of polynomial Artin--Schreier $L$-functions. 
 The \textit{polynomial} Artin--Schreier family, denoted by $\mathcal{AS}_d^0$, is defined for $(d,p)=1$ by
\begin{equation}
\mathcal{AS}_d^0 =\Big\{ f \in \F_q[x]\,:\, f(x)=\sum_{j=0} ^d a_j x^j, a_d\not = 0, a_j=0 \text{ if } j>0, p\mid j \Big\}.
\label{polyn_fam}
\end{equation}
Each curve $C_f$ with $f \in \mathcal{AS}_d^0$ has genus $\mathfrak{g} = (p-1)(d-1)/2$ and $p$--rank $0$. 

We will compute moments in the family above, and show that their behavior is given by that of random unitary matrices.
More precisely, we will prove the following theorems. Throughout, fix a non-trivial additive character $\psi$ of $\mathbb{F}_p$. Our results do not depend on the choice of $\psi$. 

\begin{thm}
\label{polyn_k}
Let $p>2$ be a odd prime and $d$ be such that  $(d,p)=1$. As $d \to \infty$, for an integer $2 \leq k< q^{1/2} $,  we have 
\begin{align*}
\frac{1}{|\mathcal{AS}_d^0|}\sum_{f\in \mathcal{AS}_d^0}\mathcal{L}\Big(\frac{1}{\sqrt{q}},f,\psi\Big)^k =& \prod_{P} \left(\frac{1}{p}\sum_{\ell=0}^{p-1} \left(1-\frac{\xi_p^\ell}{\sqrt{|P|}}\right)^{-k} \right)\\&+O \Bigg(\frac{q^{ \frac{d}{2} \left( \frac{k+1}{p}-1\right)+2}}{(1-q^{-1/2})^k} (d+1)^{k} (k+2)^{(k+1)d}k^{k-1}\Bigg),
\end{align*}
where the product on the right-hand side is over monic, irreducible polynomials; $\xi_p$ denotes a primitive $p^{\text{th}}$ root of unity; $|P|:=q^{\deg(P)}$,  and the implicit constant in the error term is independent of $k,q$ for $d$ sufficiently large.

When $k=1$, we have
\begin{align*}
\frac{1}{|\mathcal{AS}_d^0|}\sum_{f\in \mathcal{AS}_d^0}\mathcal{L}\Big(\frac{1}{\sqrt{q}},f,\psi\Big)= (1-q^{-1}) \frac{1-q^{\left(1-\frac{p}{2}\right)\left(\left\lfloor\frac{d}{p}\right \rfloor+1\right)} }{(1-q^{\left(1-\frac{p}{2}\right)})(1-q^{-\frac{p}{2}})}.
\end{align*}
\end{thm}
\begin{rmk}
In order to get an asymptotic formula above for $k \geq 2$, we need
$$(k+1) \Big( \log_q (k+2) + \frac{1}{2p} \Big)\le \frac{1}{2}-\epsilon,$$ for some $\epsilon>0$.
Note that in the expression above, the greater $\log_p q$ is, the more moments we can compute. When $q$ is a very large power of $p$, one would be able to compute roughly $p$ moments, while the case $q=p$ would allow for a more restricted range of moments. 
\end{rmk}

We also consider the second moment in absolute value for the polynomial family, which is the more standard moment to consider in the case of a family with unitary symmetry. In this case, we obtain an exact formula with a lower order term of size $d q^{\frac{d}{p}-\frac{d}{2}}$ as follows.

For $m \in \Z$, let  $[m]_p$ denote the element of $\{0,1,\dots,p-1\}$ such that $m \equiv [m]_p \pmod{p}$. 
\begin{thm}\label{thm_absvalue}
Let $p>2$ be an odd prime and $d$ be such that $(d,p)=1$.  We have
\begin{align*}\frac{1}{|\mathcal{AS}_d^0|}\sum_{f\in \mathcal{AS}_d^0}\left|\mathcal{L}\left(\frac{1}{\sqrt{q}},f,\psi\right)\right|^2 =&\frac{(1-q^{1-p})d}{(1-q^{1-\frac{p}{2}})^2} -\frac{2pq^{1-\frac{p}{2}} (1-q^{-\frac{p}{2}})}{  (1-q^{1-\frac{p}{2}})^3}\\
&+dq^{(d+p-[d]_p)\left(\frac{1}{p}-\frac{1}{2} \right)}\left(1-\frac{1}{q}\right)\left(\frac{1}{(1-q^{1-\frac{p}{2}})^2}-\frac{[d]_p}{p(1-q^{1-\frac{p}{2}})} \right)\\&+C_d\frac{q^{(d-1)\left(\frac{1}{p}-\frac{1}{2} \right)}}{p},
\end{align*}
where $C_d$ is a constant depending solely on $[d]_p$. More precisely, 
\begin{align*}C_d=& \frac{1-q^{-\frac{p}{2}}}{1-q^{1-\frac{p}{2}}} S_2(d-1,q^{\frac{1}{2}-\frac{1}{p}})+S_2(d-1,q^{\frac{1}{2}-\frac{1}{p}})+\frac{q-1}{pq}S_2(d-1,q^{\frac{1}{2}-\frac{1}{p}})(-1-[d]_p)\\&-\frac{q-1}{pq}S_2(-1,q^{\frac{1}{2}-\frac{1}{p}})S_1(d,q^{\frac{1}{2}-\frac{1}{p}})+\frac{q-1}{pq}S_3(d-1,q^{\frac{1}{2}-\frac{1}{p}})-\frac{q-1}{pq}S_3(d-2,q^{\frac{1}{2}-\frac{1}{p}}),
 \end{align*}
where the formula for $S_\ell(n,x)$ is given by Lemma \ref{lem:S}.  \end{thm}

The proofs of Theorems \ref{polyn_k} and \ref{thm_absvalue} use different techniques. The proof of Theorem \ref{polyn_k} has as starting point the relationship between the $L$-function of an Artin--Schreier curve and the $L$-function of a multiplicative character associated to each such curve, described explicitly in \cite{Ent} and \cite{EntPir}. One can roughly express the $k^{\text{th}}$ moment in terms of the $k^{\text{th}}$ moment of multiplicative characters \textit{of order $p$} modulo $x^{d+1}$. Computing moments of $L$-functions associated to fixed order characters is generally a difficult problem; moments of quadratic $L$-functions over function fields are relatively well-understood (see, for example, \cite{ak, Florea1,Florea2,Florea3}). Some partial results are known towards moments of cubic $L$-functions (see \cite{DFL}), and much less is known about higher order characters. However, in the case of Artin--Schreier $L$-functions, the multiplicative characters of order $p$ under consideration are modulo $x^d$, which is special. In this case, one can use the dichotomy exploited by Keating and Rudnick in \cite{Keating-Rudnick-moebius}, to express sums of the $k^{\text{th}}$ divisor function in arithmetic progressions in terms of short interval sums over function fields. One can then use strong results about the sum of the generalized divisor function in short intervals over function fields due to Sawin \cite{sawin}.

Proving Theorem \ref{thm_absvalue} requires different ideas, as one cannot rely on results about the divisor function in short intervals in this case. Instead, we use the approximate functional equation to write the absolute value squared of the $L$-function in terms of sums of length roughly $d$ (note that the $L$-function is a polynomial of degree approximately $d$), and then we use orthogonality relations for the sums over additive characters as in the work in \cite{BDFLS12}. 
We remark that in Theorem \ref{thm_absvalue} we computed a specific lower order term  of size  $X^\frac{2-p}{2(p-1)}$ where $X$ is roughly the size of the family, that is, $X=q^{d(1-\frac{1}{p})}$.

The paper is organized as follows. In Section \ref{section:background} we provide some background on Artin--Schreier $L$-functions and gather the results we will need from the work in \cite{EntPir}. In Sections \ref{section:proof1} and \ref{section:proof2} we prove Theorems \ref{polyn_k} and \ref{thm_absvalue} respectively. We finally check that our results match the Random Matrix Theory predictions in Section \ref{section:rmt}.

\section*{Acknowledgements}
M. Lalin thanks Julio Andrade, Alina Bucur, Chantal David, Brooke Feigon, and Jonathan Keating  for preliminary discussions (which motivated this paper) about moments of Artin--Schreier  $L$-functions which took place mainly at the American Institute of Mathematics during the workshop ``Arithmetic Statistics over finite fields and function fields'' in 2014. The study of this problem was further motivated by a talk Alexei Entin gave at the workshop ``Moments of $L$-functions'' held at the University of Northern British Columbia in August 2022, which the first and third author attended. The authors also thank the referee for the careful reading of the paper and for the numerous helpful comments and suggestions.
A substantial part of this work  was completed during the   ``Women in Numbers 6'' workshop at the Banff International Research Station for Mathematical Innovation and Discovery and the authors would like to thank the organizers and BIRS for the excellent working conditions. 
This work is supported by the National Science Foundation [DMS-2101769 to AF], the Natural Sciences and Engineering Research Council of Canada [Discovery Grant 355412-2022 to ML], the Fonds de recherche du Qu\'ebec - Nature et technologies [Projet de recherche en \'equipe 300951 to ML]. %

\section{Artin--Schreier curves and $L$-functions}
\label{section:background}
Here, we will give some basic properties of Artin--Schreier $L$-functions and their associated characters. 

\subsection{Some generalities of function field arithmetics}

We first introduce some notation and basic objects of study. Let $\mathcal{M}$ denote the set of monic polynomial in $\mathbb{F}_q[x]$, $\mathcal{M}_n$ the set of monic polynomials of degree $n$ in $\mathbb{F}_q[x]$, and $\mathcal{M}_{\leq n}$ the set of monic polynomials of degree less than or equal to $n$. For $f$ a polynomial in $\mathcal{M}$, let $d_k(f)$ denote the $k^{\text{th}}$ divisor function (i.e., $d_k(f) = \sum_{\substack{f_1 \cdots f_k = f\\f_j \in \mathcal{M}}}1$). This is extended for non-zero polynomials in $\mathbb{F}_q[x]$ by $d_k(cf):=d_k(f)$ for $c \in \F_q^*$.

The zeta-function of $\mathbb{F}_q[x]$ is defined by
$$\zeta_q(s) = \sum_{f \in \mathcal{M}} \frac{1}{|f|^s},$$ for $\re(s)>1$. By counting monic polynomials of a fixed degree, it follows that
$$\zeta_q(s) = \frac{1}{1-q^{1-s}},$$ and this provides a meromorphic continuation of $\zeta_q(s)$, with a simple pole at $s=1$. Making the change of variables $u=q^{-s}$, the zeta-function becomes
$$\mathcal{Z}(u) = \sum_{f \in \mathcal{M}} u^{\deg(f)} = \prod_P \Big(1-u^{\deg(P)} \Big)^{-1},$$ for $|u|<1/q$, where the Euler product above is over monic, irreducible polynomials. One then obtains the expression
$$\mathcal{Z}(u) = \frac{1}{1-qu},$$ for the zeta-function in the whole complex plane, having a simple pole at $u=\frac{1}{q}$.

Throughout the proof of Theorem \ref{thm_absvalue}, we use Perron's formula over function fields.
Namely, if $\mathcal{A}(u) = \sum_{f \in \mathcal{M}} a(f) u^{\deg(f)}$ is absolutely convergent in $|u| \leq r<1$, then 
\begin{equation} 
\label{perron}
\sum_{f \in \mathcal{M}_n} a(f) = \frac{1}{2\pi i} \oint_{|u|=r} \frac{ \mathcal{A}(u)}{u^{n+1}} \, du , \, \, \,  \sum_{f \in \mathcal{M}_{\leq n}} a(f) = \frac{1}{2\pi i} \oint_{|u|=r} \frac{ \mathcal{A}(u)}{u^{n+1}(1-u)} \, du .
\end{equation}

\subsection{Artin--Schreier $L$-functions}
We will consider curves given by the affine equation
$$C_f: y^p-y = f(x),$$
where $f(x) \in \mathbb{F}_q[x]$ is a polynomial of degree $d$ and $(d,p)=1$, together with the automorphism $y \mapsto y+1$. Recall the definition \eqref{def_l} of the $L$-function $\mathcal{L}(u,f,\psi)$, where $\psi$ is a non-trivial additive character of $\mathbb{F}_p$. We also have the following Euler product for the $L$-function:
$$\mathcal{L}(u,f,\psi) = \prod_P \Big(1- \psi_f(P) u^{\deg(P)} \Big)^{-1},$$  where
\[\psi_f(P)=\sum_{\substack{\alpha \in \F_{q^{\deg P}} \\ P(\alpha)=0}} \psi(f(\alpha))= \psi(\tr_{q^{\deg P}/p} f(\alpha)) \mbox{ for any } \alpha \mbox{ a root of } P,\] and where  $\tr_{q^n/p}: \mathbb{F}_{q^n} \to \mathbb{F}_p$ is the absolute trace map. We
extend $\psi_f(F)$ to a completely multiplicative function for any element of  $\F_q(x)$ by setting $\psi_f(0)=0$ and $\psi_f(c)=1$ for any non-zero $c \in \F_q$.

On the other hand, we also have that the $L$-function is a polynomial of degree $d-1$, so we can write
$$ \mathcal{L}(u,f,\psi) = \sum_{j=0}^{d-1} a_j (f,\psi) u^j,$$ where 
\[a_j(f,\psi)=\sum_{F \in \mathcal{M}_j} \psi_f(F).\]

The functional equation of $\mathcal{L}(u,f,\psi)$ is given by
\begin{equation} \mathcal{L}(u,f,\psi) = \epsilon(f)(qu^2)^{\frac{d-1}{2}} \mathcal{L} \Big( \frac{1}{qu},f,\overline{\psi}\Big). \label{fe}
\end{equation} 
(See \cite[Section 3]{RLW}.)

We will now explain how to associate a multiplicative character to each Artin--Schreier $L$-function. Before that, we quickly recall some basic facts about Dirichlet characters over $\mathbb{F}_q[x]$ and their $L$-functions.

\subsection{Multiplicative characters and their $L$-functions}

Let $Q(x)$ denote a monic polynomial in $\mathbb{F}_q[x]$. A Dirichlet character modulo $Q$ is defined to be a character of the multiplicative group $(\mathbb{F}_q[x]/Q)^{\times}$, extended to a completely multiplicative function by $\chi(g) = 0$ for any $(g,Q) \neq 1$ and $\chi(g) = \chi( g \pmod Q)$ if $(g,Q)=1$.

A Dirichlet character is even if $\chi(cF)= \chi(F)$ for any $0 \neq c \in \mathbb{F}_q$, and odd otherwise. A character is primitive if there is no proper divisor $Q_1|Q$ such that $\chi(g)=1$ whenever $(g,Q)=1$ and $g \equiv 1 \pmod{Q_1}$.

If $Q \in \mathcal{M}$, and $H$ denotes a subset of the group of characters modulo $Q$, we denote by $H^{\text{pr}}$ the set of primitive characters in $H$. 

The $L$-function associated to a Dirichlet character $\chi$ modulo $Q$ is given by
$$\mathcal{L}(u,\chi) = \sum_{F \in \mathcal{M}} \chi(F) u^{\deg(F)} = \prod_P \Big(1-\chi(P) u^{\deg(P)} \Big)^{-1},$$ where the product is over monic irreducible polynomials $P$.
Using orthogonality of characters, it follows that if $\chi$ is non-principal, $\mathcal{L}(u,\chi)$ is a polynomial of degree at most $\deg(Q)-1$. 

If $\chi$ is a primitive character, then the Riemann hypothesis for $\mathcal{L}(u,\chi)$ is true, and we can write
$$\mathcal{L}(u,\chi) = (1-u)^{\epsilon(\chi)} \prod_{j=1}^{\deg(Q)-1-\epsilon(\chi)} \Big(1 - u \sqrt{q} \rho_j \Big),$$ where $|\rho_j|=1$ are the normalized inverse roots of $\mathcal{L}(u,\chi)$, and where $\epsilon(\chi)=1$ if $\chi$ is even and $\epsilon(\chi)=0$ otherwise. 

\subsection{Multiplicative characters associated to Artin--Schreier curves}
Here, we gather a few results we need from the work of Entin \cite{Ent} and Entin--Pirani \cite{EntPir}, which relate the Artin--Schreier $L$-functions to Dirichlet $L$-functions. 
For $(d,p)=1$, let
\[\mathcal{F}_d =\Big\{ f \in \F_q[x]\,:\, f(x)=\sum_{j=0} ^d a_j x^j, a_d\not = 0, a_j=0 \text{ if }p\mid j \Big\}.\]
Note that we have
\[|\mathcal{F}_d|=(q-1)q^{d-\left \lfloor\frac{d}{p}\right\rfloor-1}.\]
Using \eqref{polyn_fam}, we can write
\begin{equation}\label{aodd}
\mathcal{AS}_d^0=\bigsqcup_{b \in \F_q} \{f+b\,:\, f \in \mathcal{F}_d\}.
\end{equation}
We have the following result from \cite{EntPir}. 
\begin{lem}\cite[Lemma 2.1]{EntPir}
For $f \in \F_q(x)\setminus \F_q$ which is not of the form $f=h^p-h$ for $h\in \F_q(x)$, we have $\mathcal{L}(u,f+b,\psi)=\mathcal{L}(\psi(\tr_{q/p}(b))\cdot u, f, \psi)$.
\label{shift}
\end{lem}

For $c \in \F_q[x]$, if $x|c$, we put $\chi_f(c)=0$; otherwise, let  $$\chi_f(c):= \psi\Big(\tr_{q/p} \Big(\sum_{c(\alpha)=0} f(\alpha)\Big)\Big).$$
We also have the following.
\begin{prop}\cite[Lemmata $7.1$, $7.2$]{Ent}
\label{entin_lemma}
Assume that $(d,p)=1$, and $f \in \mathcal{F}_d$. Then
\begin{itemize}
 \item $\chi_f$  is a primitive Dirichlet character modulo $x^{d+1}$  of order $p$. (In particular $\chi_f$ is even since $(p,q-1) = 1$, and $|\F_q^*|=q-1$.) 
 
 \item The map 
 $$\mathcal{F}_d \rightarrow \{\text{primitive Dirichlet characters modulo }x^{d+1} \text{ of order }p \}$$
 given by $f \mapsto \chi_f$ is a bijection. 
 \item  $\mathcal{L}(u,\chi_f) = (1-u) \mathcal{L}(u,f,\psi)$.
\end{itemize}

\end{prop}

Following \cite{EntPir}, for an abelian group $A$ and a group of characters $B \subseteq A^{*}$, let
$$B^{\perp} = \{ a\in A \,:\, \chi(a)=1 \text{ for all } \chi \in B \}.$$
The orthogonality relations imply that
\begin{equation} 
\label{orthog}
 \frac{1}{|B|} \sum_{\chi \in B} \chi(a) =
\begin{cases}
1 & \mbox{ if } a \in B^{\perp}, \\ 
0 & \mbox{ if } a \notin B^{\perp}.
\end{cases} 
\end{equation}

\subsection{Sums involving roots of unity}
Here, we will prove certain results about sums involving roots of unity, which we will use repeatedly throughout the paper. Let
$$S_{\ell}(n,x) = \sum_{j=0}^{p-1} \frac{\xi_p^{-nj}}{(1-\xi_p^j x)^\ell},$$
where $\xi_p$ is a non-trivial $p^{\text{th}}$ root of unity in $\C$. 
We will obtain formulas for $S_{\ell}(n,x)$ for $\ell=1,2,3$.
For $m \in \Z$, recall that  $[m]_p$ denotes the element of $\{0,1,\dots,p-1\}$ such that $m \equiv [m]_p \pmod{p}$. 
\begin{lem} \label{lem:S} 
For $|x|>1$ and $n \in \Z$, we have
$$S_1(n,x) =\frac{p x^{[n]_p}}{1 - x^p}, $$
$$S_2(n,x) =px^{[n+1]_p - 1}\left(\frac{[n+1]_p}{1-x^{p}} + \frac{p x^p}{(1-x^{p})^2}\right),$$
and
$$S_3(n,x) = \frac{p x^{[n+2]_p - 2}}{2} \left(\frac{[n+2]_p ([n+2]_p -1)}{1-x^{p}} + \frac{p (2 [n+2]_p + p - 1) x^p}{(1-x^{p})^2} + \frac{2 p^2 x^{2p}}{(1-x^{p})^3}\right).$$
\end{lem}
\begin{proof} We have
\begin{align*}S_1(n,x) =&\sum_{j=0}^{p-1} \frac{1}{(1-\xi_p^j x)\xi_p^{nj}}=-\frac{1}{x}\sum_{k=0}^\infty x^{-k}  \sum_{j=0}^{p-1}\xi_p^{-(n+1+k)j}\\
=&-\frac{p}{x}\sum_{\substack{k=0\\k\equiv -n-1\pmod{p}}}^\infty x^{-k} = -\frac{px^{-[-n-1]_p-1}}{1-x^{-p}},
\end{align*}
since the inner sum in the first line is equal to $p$ when $n+1+k$ is divisible by $p$, and 0 otherwise.

Notice that $[-n-1]_p+1=p-[n]_p$, giving 
\[S_1(n,x)=-\frac{px^{[n]_p}}{x^{p}-1}.\]
The expressions of $S_2(n,x)$ and $S_3(n,x)$ can be obtained by using the fact that $$S_{\ell+1}(n,x) = \frac{1}{\ell} \frac{\partial}{\partial x} S_\ell(n+1,x)$$ when $\ell > 0$.
\end{proof}
We will also need the case $x=1$. More precisely, we prove the following result.
\begin{lem} We have 
\[ \sum_{j=1}^{p-1}  \frac{\xi_p^{-nj}}{(1-\xi_p^j)}=\frac{p-1}{2}-[n]_p.\]
\end{lem}
\begin{proof}
Notice that for $0< a \leq p-1$,
\[\sum_{j=1}^{p-1}  \frac{1-\xi_p^{aj}}{(1-\xi_p^j) }=\sum_{j=1}^{p-1} (1+\xi_p^j+\cdots+\xi_p^{(a-1)j})=p-a.\]
By looking at the logarithmic derivative at $x=1$ of $x^{p-1}+\cdots+x+1=\prod_{j=1}^{p-1}(x-\xi_p^j)$, we have
\[\sum_{j=1}^{p-1}  \frac{1}{(1-\xi_p^j) }=\frac{p-1}{2}.\]
We conclude by combining the above equations. 

\end{proof}

\subsection{Some combinatorial identities}

Here we consider some combinatorial identities that will be needed in the proof of Theorem \ref{polyn_k}. 
\begin{lem} \label{lem:Knuth}
 Let $r,t,m\geq 0$ be integers and $s$ real. Then 
 \begin{equation}\label{eq:Knuth}
 \sum_{j=0}^r \binom{r-j}{m}\binom{s}{j-t} (-1)^{j-t} =\binom{r-t-s}{r-t-m}.
 \end{equation}
\end{lem}
This can be found in Equation (24) in \cite[1.2.6]{Knuth}. 

\begin{lem} \label{lem:friendly}
 If $m, r, t\in \Z$ with $m>0$ and $r, t\geq 0$, then
 \begin{equation}\label{eq:friendly}
 m\binom{m+t}{t}\sum_{j=0}^{r}\binom{j+t}{t}\binom{t}{r-j} \frac{(-1)^{r-j} }{m+r-j}=\binom{r+m+t}{t}.
\end{equation}
 \end{lem}
\begin{proof}
First notice that 
\[\binom{j+t}{t}=(-1)^j \binom{-t-1}{j}.\]
Therefore, it suffices to prove 
\[\binom{-t-1}{m}\sum_{j=0}^{r}\binom{-t-1}{j}\binom{t}{r-j} \frac{m}{m+r-j}=\binom{-t-1}{m+r}.\]
Notice that 
\[\sum_{j=0}^{r}\binom{-t-1}{j}\binom{t}{r-j} x^{m-1+r-j}=[x^{m-1}(1+T)^{-t-1}(1+xT)^{t}]_{T^r},\]
where the notation $[\cdot]_{T^r}$ indicates the coefficient of $T^r$ in the expression inside the brackets. 
We integrate to get 
\[\sum_{j=0}^{r}\binom{-t-1}{j}\binom{t}{r-j} \frac{1}{m+r-j} =\Big [(1+T)^{-t-1} \int_0^1x^{m-1}(1+xT)^{t} dx \Big]_{T^r}. \]
By doing integration by parts repeatedly, we obtain
\begin{align*}(1+T)^{-t-1}\int_0^1x^{m-1}(1+xT)^{t} dx =&t!(m-1)!\sum_{h=0}^{m-1}\frac{(-1)^{h}}{(h+t+1)!(m-1-h)!}
\frac{(1+T)^h}{T^{h+1}}\\&+\frac{t!(m-1)! (-1)^m}{(m+t)!T^m(1+T)^{t+1}}.
 \end{align*}
 The only term above contributing to a monomial $T^r$  is the last one, with coefficient 
 \[\frac{t!(m-1)! (-1)^m}{(m+t)!}\binom{-t-1}{m+r}=\frac{1}{m}\binom{-t-1}{m}^{-1}\binom{-t-1}{m+r}.\]
This proves the result. 
 
\end{proof}

\section{Proof of Theorem \ref{polyn_k}}
\label{section:proof1}
Before proving Theorem \ref{polyn_k}, we first state some results that we will need and prove some preliminary lemmas. 

The following estimate for sums of the generalized divisor function in short intervals is due to Sawin \cite{sawin}, and will be crucial in our computations.
\begin{prop}(\cite[Theorem 1.1.]{sawin})\label{sawin_prop}
For natural numbers $n$, $h$, $k$ with $h < n$ and $f$ a monic polynomial of degree $n$ in $\F_q[T]$, we have
\begin{align*}
    \left| \sum_{\substack{g \in \F_q[T] \\ \deg g < h}} d_k (f+g) - \binom{n+k-1}{k-1} q^h \right| \le 3 \binom{n+k-1}{k-1} (k+2)^{2n-h} q^{\frac{1}{2} \left( h + \left\lfloor\frac{n}{p} \right\rfloor - \left\lfloor \frac{n-h}{p} \right\rfloor + 1 \right)} .
\end{align*}
\end{prop}
We will now prove the following lemma.
\begin{lem}\label{mainlemma}
Let $n>d\geq 0$ and let $A(x) \in \mathbb{F}_q[x]$ such that $A(0) \neq 0$ and $\deg(A) \leq d$. Then
\begin{align*}
\sum_{\substack{\deg(F)=n \\F(x) \equiv A(x) \pmod{x^{d+1}}}}d_k(F) =& \sum_{j=0}^{n-d-1} \binom{n-j+k-1}{k-1}\binom{k}{j} (-1)^j q^{n-d-j}\\&+O\Bigg(\binom{k-1}{n-d-1} \frac{d-\deg(A)+k}{n-\deg(A)}\binom{d-\deg(A)+k-1}{k-1}d_k(A)\\&+
(k+2)^{n+d}
q^{\frac{1}{2}\left((n-d)\left(1+\frac{1}{p}\right)+2  \right)} (n+k-1)^{k-1} \Bigg),
\end{align*}
and the implied constant in the error term above does not depend on $k$ and $q$.
\end{lem}

\begin{proof}
Let 
$$A(x) = \sum_{j=0}^d a_j x^j,$$ where $a_j \in \mathbb{F}_q$, $a_0 \neq 0$ (note that $a_d$ can be $0$).
Also note that we can assume, without loss of generality, that $a_0=1$. Otherwise, we can rewrite the sum as
$$\sum_{\substack{\deg(F)=n \\F(x) \equiv A(x) \pmod{x^{d+1}}}}d_k(F) = \sum_{\substack{\deg(F) = n \\ a_0^{-1} F(x) \equiv A_1(x) \pmod{x^{d+1}}}} d_k(F),$$ where $A_1(x) = a_0^{-1} A(x)$. Since $d_k(F) = d_k(cF)$ for $c \in \mathbb{F}_q^{*}$, we easily see that 
$$\sum_{\substack{\deg(F)=n \\F(x) \equiv A(x) \pmod{x^{d+1}}}}d_k(F) = \sum_{\substack{\deg(F)=n \\F(x) \equiv A_1(x) \pmod{x^{d+1}}}}d_k(F) ,$$ and $A_1(x)$ has the property that its constant coefficient is $1$.

Now since $F(x) \equiv A(x) \pmod{x^{d+1}}$, we write
$$F(x) = f_n x^n+ \cdots +f_{d+1}x^{d+1}+a_d x^d+ \cdots+a_1x+1,$$
where $f_n\not =0$. 
Now let 
$$F^{*}(x) = x^n F\Big( \frac{1}{x}\Big)$$ be the reverse polynomial of $F$. Then we have
$$F^{*}(x) = x^n+a_1x^{n-1}+ \cdots+a_dx^{n-d}+f_{d+1} x^{n-d-1}+\cdots+f_n,$$ and $\deg(F^{*})=n$. Note that we can write
$$F^{*}(x) = x^{n-\deg(A)} A^{*}(x) + g(x),$$ where $g(x)$ varies over polynomials of degree less than $n-d$ and such that $g(0)\not =0$. Also note that for $F(x)$ such that $F(0) \neq 0$, we have $d_k(F) = d_k(F^{*})$. 
Hence we rewrite 
$$ \sum_{\substack{\deg(F) =n \\F(x) \equiv A(x) \pmod{x^{d+1}}}}d_k(F) = \sum_{\substack{\deg(g)<n-d\\g(0)\not = 0 }}d_k(g(x)+x^{n-\deg(A)}A^*(x)).$$

We have 
\begin{align}\label{eq:genalphabeta}
 &\sum_{\substack{\deg(g)<n-d} }d_k(g(x)+x^{n-\deg(A)}A^*(x))\nonumber\\&=d_k(x^{n-\deg(A)}A^*(x))+\sum_{j=0}^{n-d-1}d_k(x^j)\sum_{\substack{\deg(g)<n-d-j\\g(0)\not =0} }d_k(g(x)+x^{n-j-\deg(A)}A^*(x))\nonumber\\
 &=\binom{n-\deg(A)+k-1}{k-1} d_k(A^*(x))+\sum_{j=0}^{n-d-1}\binom{j+k-1}{k-1}\sum_{\substack{\deg(g)<n-d-j\\g(0)\not =0} }d_k(g(x)+x^{n-j-\deg(A)}A^*(x)),
\end{align}
where the term $d_k(x^{n-\deg(A)}A^*(x))$ is accounting for the polynomial $g=0$, which we consider to have degree $-\infty$. 

For $n>0$, we let  \[\gamma_{n}:=\sum_{\substack{\deg(g)<n} }d_k(g(x)+x^{n+d-\deg(A)}A^*(x))\quad \mbox{ and }\quad\beta_{n}:=\sum_{\substack{\deg(g)<n\\g(0)\not =0} }d_k(g(x)+x^{n+d-\deg(A)}A^*(x)).\]
Set also by convention (to take into account the polynomial $g=0$),  $\gamma_0=d_k(x^{d-\deg(A)}A^*(x))$ and 
$\beta_0=d_k(A^*(x)) $. 
Then, \eqref{eq:genalphabeta} can be written as 
\begin{equation}\label{eq:genalphabeta2}
\gamma_{n-d}=\binom{n-\deg(A)+k-1}{k-1}  \beta_0+ \sum_{j=0}^{n-d-1}\binom{j+k-1}{k-1}\beta_{n-d-j}.
\end{equation}
We claim that  for $n>0$, 
\begin{equation}\label{eq:alphabeta}\beta_n=(-1)^n \binom{k-1}{n-1} \frac{d-\deg(A)+k}{d-\deg(A)+n}\gamma_0+\sum_{j=0}^{n-1} \binom{k}{j} (-1)^j \gamma_{n-j}.\end{equation}
Indeed,  for $n=1$ we have from \eqref{eq:genalphabeta2} that $\gamma_1=\binom{d+1-\deg(A)+k-1}{k-1}  \beta_0+\beta_1$, which gives $\beta_1=\gamma_1- \frac{d-\deg(A)+k}{d-\deg(A)+1}\gamma_0$. We proceed by induction. Suppose that \eqref{eq:alphabeta} is true for all positive integers up to $n$. By  \eqref{eq:genalphabeta2}, we have 
\[\gamma_{n+1}=\binom{n+d+1-\deg(A)+k-1}{k-1}  \beta_0+ \sum_{j=0}^{n}\binom{j+k-1}{k-1}\beta_{n+1-j}\]
and
\begin{align*}\beta_{n+1}=&\gamma_{n+1}-\binom{n+d-\deg(A)+k}{k-1}  \beta_0-\sum_{j=1}^{n}\binom{j+k-1}{k-1}\beta_{n+1-j}\\
=&\gamma_{n+1}-\binom{n+d-\deg(A)+k}{k-1} \binom{d-\deg(A)+k-1}{k-1}^{-1} \gamma_0\\& -(d-\deg(A)+k)\sum_{j=1}^{n}\binom{j+k-1}{k-1}\binom{k-1}{n-j} \frac{(-1)^{n+1-j} }{d-\deg(A)+n+1-j}\gamma_0\\&-\sum_{j=1}^{n}\binom{j+k-1}{k-1}
\sum_{h=0}^{n-j} \binom{k}{h} (-1)^h \gamma_{n+1-j-h}.
\end{align*}
Applying Lemma~\ref{lem:friendly}, we have 
\begin{align*}
\beta_{n+1}=&\gamma_{n+1}+(-1)^{n+1}\binom{k-1}{n}\frac{d-\deg(A)+k}{d-\deg(A)+n+1}\gamma_0\\& -\sum_{j=1}^{n}\binom{j+k-1}{k-1}
\sum_{\ell=j}^{n} \binom{k}{\ell-j} (-1)^{\ell-j} \gamma_{n+1-\ell }, \quad \mbox{where $\ell=j+h$},
\\
=&\gamma_{n+1}+(-1)^{n+1}\binom{k-1}{n}\frac{d-\deg(A)+k}{d-\deg(A)+n+1}\gamma_0\\&-\sum_{\ell=1}^{n}
\sum_{j=1}^{\ell}\binom{j+k-1}{k-1}
 \binom{k}{\ell-j} (-1)^{\ell-j} \gamma_{n+1-\ell }
\\
=&\gamma_{n+1}+(-1)^{n+1}\binom{k-1}{n}\frac{d-\deg(A)+k}{d-\deg(A)+n+1}\gamma_0\\&-\sum_{\ell=1}^{n}
\sum_{m=0}^{\ell-1}\binom{\ell+k-1-m}{k-1}
 \binom{k}{m} (-1)^{m} \gamma_{n+1-\ell },  \quad \mbox{where $m=\ell-j$},
\\
=&\gamma_{n+1}+(-1)^{n+1}\binom{k-1}{n}\frac{d-\deg(A)+k}{d-\deg(A)+n+1}\gamma_0\\&+\sum_{\ell=1}^{n}
\binom{k}{\ell}(-1)^\ell\gamma_{n+1-\ell },
\end{align*}
where in the last equality we applied Lemma~\ref{lem:Knuth}. 

This concludes the induction proving \eqref{eq:alphabeta}. 

Putting \eqref{eq:genalphabeta} together with Proposition \ref{sawin_prop} in the case $h=n-d-j$ and using the fact that $A^{*}(x)x^{n-\deg(A)-j}$ is a monic polynomial of degree $n-j$, it follows from \eqref{eq:alphabeta} that
\begin{align*}&\sum_{\substack{\deg(F)=n \\F(x) \equiv A(x) \pmod{x^{d+1}}}}d_k(F) = \sum_{\substack{\deg(g)<n-d\\g(0)\not=0} }d_k(g(x)+x^{n-\deg(A)}A^*(x))= \beta_{n-d}\\=&(-1)^n \binom{k-1}{n-d-1} \frac{d-\deg(A)+k}{n-\deg(A)}\gamma_0+\sum_{j=0}^{n-d-1} \binom{k}{j} (-1)^j \gamma_{n-d-j}\\
=&\sum_{j=0}^{n-d-1} \binom{n-j+k-1}{k-1}\binom{k}{j} (-1)^j q^{n-d-j}\\&+O\Bigg(\binom{k-1}{n-d-1} \frac{d-\deg(A)+k}{n-\deg(A)}\binom{d-\deg(A)+k-1}{k-1}d_k(A)\\&+ \sum_{j=0}^{n-d-1} \binom{n-j+k-1}{k-1} \binom{k}{j}  (k+2)^{n+d-j}q^{\frac{1}{2}\left(n-d -j+ \left\lfloor \frac{n-j}{p} \right\rfloor - \left\lfloor \frac{d}{p} \right\rfloor +1  \right)}\Bigg).
\end{align*}
Now notice that the sum over $j$ in the error term above is 
\begin{align*}
&\sum_{j=0}^{n-d-1} \binom{n-j+k-1}{k-1} \binom{k}{j}  (k+2)^{n+d-j}q^{\frac{1}{2}\left(n-d -j+ \left\lfloor \frac{n-j}{p} \right\rfloor - \left\lfloor \frac{d}{p} \right\rfloor +1  \right)} \\
&\le (k+2)^{n+d} \sum_{j=0}^{n-d-1}  (n-j+k-1)^{k-1} k^j  (k+2)^{-j }
q^{\frac{1}{2}\left((n-d -j)\left(1+\frac{1}{p}\right)+2  \right)} \\
&\le (k+2)^{n+d}
q^{\frac{1}{2}\left((n-d)\left(1+\frac{1}{p}\right)+2  \right)} \sum_{j=0}^{n-d-1}  (n-j+k-1)^{k-1}
q^{-\frac{j}{2}\left(1+\frac{1}{p} \right)}\\
& \leq (k+2)^{n+d}
q^{\frac{1}{2}\left((n-d)\left(1+\frac{1}{p}\right)+2  \right)}  (n+k-1)^{k-1} \frac{1}{1-q^{- \frac{1}{2}\left(1+\frac{1}{p}\right)}}\\ & \leq (k+2)^{n+d}
q^{\frac{1}{2}\left((n-d)\left(1+\frac{1}{p}\right)+2  \right)} \frac{(n+k-1)^{k-1}}{1-3^{-\frac{1}{2}}}.
\end{align*}
This finishes the proof of the statement.
\end{proof}

Now for $\ell \pmod{p}$, we let
\begin{equation*}
\alpha_k(\ell)=\sum_{b \in \F_q}\frac{\psi(\tr_{q/p}(b))^{\ell}}{\left(1-\frac{\psi(\tr_{q/p}(b))}{\sqrt{q}}\right)^k}.
\end{equation*}
We will prove the following. 
\begin{lem}
\label{alpha}
For $\ell \pmod{p}$, we have 
\begin{align*}
    \alpha_k(\ell) = \frac{q}{p} \sum_{j=0}^{p-1} \frac{\xi_p^{j \ell}}{\left( 1 -q^{-1/2}\xi_p^j  \right)^{k}}.
\end{align*}
\end{lem} 
\begin{proof}
We have
\begin{align*}
    \alpha_k(\ell) &= \sum_{b \in \F_q} \psi(\tr_{q/p}(b))^{\ell} \left(1-\frac{\psi(\tr_{q/p}(b))}{\sqrt{q}}\right)^{-k} \\
    &= \sum_{b \in \F_q} \psi(\tr_{q/p}(b))^{\ell} \sum_{h=0}^\infty \binom{-k}{h} \left(-\frac{\psi(\tr_{q/p}(b))}{\sqrt{q}}\right)^{h}
\end{align*}
by the binomial theorem. Therefore,
\begin{align*}
    \alpha_k(\ell) &= \sum_{h=0}^\infty \binom{-k}{h} (-q^{-1/2})^{h} \sum_{b \in \F_q} \psi(\tr_{q/p}(b))^{\ell + h}.
\end{align*}
The inner sum above  is equal to $q$ when $\ell + h$ is divisible by $p$. Otherwise, the inner sum is equal to zero since $\psi$ is a nontrivial additive character. Therefore,
\begin{align*}
    \alpha_k(\ell) &= \frac{q}{p} \sum_{h=0}^\infty \binom{-k}{h} (-q^{-1/2})^{h} \sum_{j=0}^{p-1} \xi_p^{j(\ell+h)}.
\end{align*}
We switch the order of summation and use the binomial theorem again to conclude that
\begin{align*}
    \alpha_k(\ell) &= \frac{q}{p} \sum_{j=0}^{p-1} \xi_p^{j\ell} \sum_{h=0}^\infty \binom{-k}{h} \left( -q^{-1/2}\xi_p^j\right)^{h} = \frac{q}{p} \sum_{j=0}^{p-1} \frac{\xi_p^{j\ell}}{  \left( 1 -q^{-\frac{1}{2}} \xi_p^{j}  \right)^{k}}.
\end{align*}
\end{proof}
We are now ready to begin the proof of Theorem \ref{polyn_k}.
\begin{proof}[Proof of Theorem \ref{polyn_k}]
 Using \eqref{aodd}, Lemma \ref{shift}, and Proposition \ref{entin_lemma}, we write 
\begin{align}
\frac{1}{|\mathcal{AS}_d^0|}\sum_{f\in \mathcal{AS}_d^0}\mathcal{L}\left(\frac{1}{\sqrt{q}},f,\psi\right)^k =& 
\frac{1}{q|\mathcal{F}_d|}\sum_{b \in \F_q}\sum_{f\in \mathcal{F}_d}\mathcal{L}\left(\frac{1}{\sqrt{q}},f+b,\psi\right)^k \nonumber \\
 =& 
\frac{1}{q|\mathcal{F}_d|}\sum_{b \in \F_q}\sum_{f\in \mathcal{F}_d}\mathcal{L}\left(\frac{\psi(\tr_{q/p}(b))}{\sqrt{q}},f,\psi\right)^k \nonumber \\
=&\frac{1}{q|\mathcal{F}_d|}\sum_{b \in \F_q}\frac{1}{\left(1-\frac{\psi(\tr_{q/p}(b))}{\sqrt{q}}\right)^k}\sum_{f\in \mathcal{F}_d}\mathcal{L}\left(\frac{\psi(\tr_{q/p}(b))}{\sqrt{q}},\chi_f\right)^k \nonumber \\
=&\frac{1}{q|\mathcal{F}_d|}\sum_{b \in \F_q}\frac{1}{\left(1-\frac{\psi(\tr_{q/p}(b))}{\sqrt{q}}\right)^k}\sum_{f\in \mathcal{F}_d}\sum_{\substack{F\in \mathcal{M}_{\leq kd}\\F(0)\not =0}}\frac{d_k(F)\chi_f(F)\psi(\tr_{q/p}(b))^{\deg(F)}}{\sqrt{|F|}} \nonumber \\
=&\frac{1}{q|\mathcal{F}_d|}\sum_{f\in \mathcal{F}_d}\sum_{\substack{F\in \mathcal{M}_{\leq kd}\\F(0)\not =0}}\frac{\alpha_k(\deg(F)) d_k(F)\chi_f(F)}{\sqrt{|F|}}.\label{firstsum}
\end{align}

Interchanging the sums over $f$ and $F$, we then need to study $\sum_{f\in \mathcal{F}_d}\chi_f(F)$ for $F$ fixed. 

Let \[H_n=\{\chi \pmod{x^n} \, :\, \chi^p=1\}\]
and 
\[H_n^\text{pr}=\{\chi \in H_n \, : \,  \chi \text{ primitive}\}.\]
(We will work with both $n=d$ and $n=d+1$.)
Using Proposition \ref{entin_lemma}, there is a bijection between $\mathcal{F}_d$ and $H_{d+1}^\text{pr}$. Moreover, a character in $H_{d+1}$ that is not primitive is necessarily a character in $H_d$. Thus we have 
\[H_{d+1}^\text{pr}=H_{d+1}\setminus H_d.\]
It follows that
\begin{align*}\sum_{f\in \mathcal{F}_d}\chi_f(F)=& \sum_{\chi \in H_{d+1}^\text{pr}}\chi(F)=\sum_{\chi \in H_{d+1}}\chi(F)-\sum_{\chi \in H_{d}}\chi(F).
\end{align*}
Now, using \eqref{orthog}, we have 
\begin{align}
\label{sumf}
\sum_{f\in \mathcal{F}_d}\chi_f(F)=\begin{cases}
0 & F \not \in H_d^\bot,\\
-|H_d| & F \in H_d^\bot\setminus H_{d+1}^\bot,\\
|H_{d+1}|-|H_d|& F \in H_{d+1}^\bot.
\end{cases}
\end{align}
Notice that $H_d\subseteq H_{d+1}$ implies that $H_{d+1}^\bot\subseteq H_{d}^\bot$.

Let us compute the order of $H_n$ with $p \nmid n$. Following the proof of \cite[Lemma 7.1]{Ent},
 $|H_n|$ corresponds to $|(\F_q[x]/x^n)^\times[p]|$ and that corresponds to counting polynomials  $g(x)=\sum_{j=0}^{n-1}c_jx^j$ such that $g(x)^p\equiv 1 \pmod{x^n}$ and $c_0\not = 0$ (so that $g(x)$ is a unit). 
Taking the $p^{\text{th}}$ power, we see that this imposes the condition $c_0^p=1$ (implying $c_0=1$) and $c_1^p=\cdots =c_{\left \lfloor \frac{n-1}{p}\right\rfloor}^p=0$
(implying $c_1=\cdots =c_{\left \lfloor \frac{n-1}{p}\right\rfloor}=0$). The total count is then 
\[|H_n| = \frac{q^{n-1}(q-1)}{q^{\left \lfloor \frac{n-1}{p}\right \rfloor}(q-1)}=q^{n-1-\left \lfloor \frac{n}{p}\right \rfloor},\]
where we have used that, since $p \nmid n$, we have $\left \lfloor \frac{n-1}{p}\right \rfloor=\left \lfloor \frac{n}{p}\right \rfloor$.

By \cite[Lemma 4.1]{EntPir}, $F(x) \in \F_q[x]$ with $F(0)\not = 0$
satisfies $\chi(F)=1$ for all $\chi \in H_{d+1}$ if and only if  $F(x) \equiv R(x^p)\pmod{x^{d+1}}$ for some $R(x) \in \F_q[x]$ with $R(0)\not = 0$ and $\deg(R)\leq \left \lfloor\frac{d}{p}\right \rfloor$. 
A similar result applies for  $F(x) \in \F_q[x]$ with $F(0)\not = 0$
that satisfies $\chi(F)=1$ for all $\chi \in H_{d}$.

Putting \eqref{firstsum}, \eqref{sumf} and the observation above together, we get that
\begin{align}\frac{1}{|\mathcal{AS}_d^0|}\sum_{f\in \mathcal{AS}_d^0}\mathcal{L}\left(\frac{1}{\sqrt{q}},f,\psi\right)^k 
=&\frac{|H_{d+1}|}{q|\mathcal{F}_d|}\sum_{\substack{F\in \mathcal{M}_{\leq kd}\\F(0)\not =0\\F(x) \in H_{d+1}^\bot}}\frac{\alpha_k(\deg(F)) d_k(F)}{\sqrt{|F|}} \nonumber \\
&-\frac{|H_d|}{q|\mathcal{F}_d|}\sum_{\substack{F\in \mathcal{M}_{\leq kd}\\F(0)\not =0\\F(x) \in H_d^\bot}}\frac{\alpha_k(\deg(F)) d_k(F)}{\sqrt{|F|}} \nonumber \\
=& \frac{|H_{d+1}|}{q|\mathcal{F}_d|} S_{k,d+1}-\frac{|H_d|}{q|\mathcal{F}_d|}S_{k,d} \nonumber \\
=& \frac{S_{k,d+1}}{q-1} -\frac{S_{k,d}}{q(q-1)}, \label{secondsum}
\end{align}
where
\begin{align*}
S_{k,d+1}=&\sum_{\substack{\deg(R)\leq \left \lfloor \frac{d}{p}\right \rfloor \\ R(0)\not =0}}
\sum_{\substack{F\in \mathcal{M}_{\leq kd}\\F(x) \equiv R(x^p) \pmod{x^{d+1}}}}\frac{\alpha_k(\deg(F))d_k(F)}{{\sqrt{|F|}}},
\end{align*}
and 
\begin{align*}
S_{k,d}=&\sum_{\substack{\deg(R)\leq \left \lfloor \frac{d-1}{p}\right \rfloor \\ R(0)\not =0}}
\sum_{\substack{F\in \mathcal{M}_{\leq kd}\\F(x) \equiv R(x^p) \pmod{x^{d}}}}\frac{\alpha_k(\deg(F))d_k(F)}{{\sqrt{|F|}}}.
\end{align*}
Notice that since $p \nmid d$, the condition $\deg(R)\leq \left \lfloor \frac{d}{p}\right \rfloor 
$ is equivalent to the condition 
$\deg(R)\leq \left \lfloor \frac{d-1}{p}\right \rfloor$.

Now note that the terms in the inner sum that satisfy $\deg(F)<d$ for $S_{k,d}$ (resp. $\deg(F)<d+1$ for $S_{k,d+1}$) have the property that $F(x)=R(x^p)$, and we can write $R(x^p)=R_0(x)^p$ by applying the Frobenius automorphism.

When $k=1$, note that we have $S_{1,d+1}=S_{1,d}$, and using \eqref{secondsum}, we have that the moment under consideration equals
\begin{align*}
\frac{1}{q}\sum_{\substack{R\in \mathcal{M}_{\leq \left \lfloor \frac{d}{p}\right \rfloor}\\ R(0)\not =0}}
\frac{\alpha_1(\deg(R)p)}{{|R|^{\frac{p}{2}}}}=\frac{\alpha_1(0)}{q}\sum_{j=0}^{\left \lfloor \frac{d}{p}\right \rfloor} \frac{q^{j-1}(q-1)}{q^{\frac{jp}{2}}} =\frac{\alpha_1(0)}{q}(1-q^{-1}) \frac{1-q^{\left(1-\frac{p}{2}\right)\left(\left \lfloor \frac{d}{p}\right \rfloor+1\right)} }{1-q^{\left(1-\frac{p}{2}\right)}}.
\end{align*}

Note that Lemmas \ref{alpha} and \ref{lem:S} imply that 
\begin{align*}
\alpha_1(0) = \frac{q}{p} \sum_{j=0}^{p-1} \frac{1}{1- \frac{\xi_p^j}{\sqrt{q}}}=
- \frac{q^\frac{3}{2}}{p} \sum_{m=0}^{p-1} \frac{\xi_p^{m}}{1-\xi_p^{m}\sqrt{q}}=
- \frac{q^\frac{3}{2}}{p}  S_1 \Big( -1 ,\sqrt{q}\Big)= \frac{q}{1-q^{-\frac{p}{2}}}.
\end{align*}

Putting the above together finishes the proof of Theorem \ref{polyn_k} in the case $k=1$. 

Now we consider $k>1$. Suppose that $n>d$. 
Then, using Lemma \ref{mainlemma}, we have that
\begin{align*}&\sum_{\substack{\deg(F)=n \\F(x) \equiv R_0(x)^p \pmod{x^{d+1}}}}d_k(F) =\sum_{j=0}^{n-d-1} \binom{n-j+k-1}{k-1}\binom{k}{j} (-1)^j q^{n-d-j}\\
&+O\Bigg(\binom{k-1}{n-d-1} \frac{d-p\deg(R_0)+k}{n-p\deg(R_0)}\binom{d-p\deg(R_0)+k-1}{k-1}d_k(R_0^p)\\&+ (k+2)^{n+d} q^{\frac{1}{2}\left((n-d)\left(1+\frac{1}{p}\right)+2  \right)}   (n+k-1)^{k-1} \Bigg)\\
&=\sum_{j=0}^{n-d-1} \binom{n-j+k-1}{k-1}\binom{k}{j} (-1)^j q^{n-d-j}\\
&+O\left(\binom{k-1}{n-d-1} \frac{d+k}{n-d}\binom{d+k-1}{k-1}d_k(R_0^p)\right.\\&+\left. (k+2)^{n+d} q^{\frac{1}{2}\left((n-d)\left(1+\frac{1}{p}\right)+2  \right)}   (n+k-1)^{k-1} \right).
\end{align*}
Notice that the  main term in the above expression is independent of $R(x^p)$. Moreover, $d_k(F)=d_k(cF)$
for any $c \in \F_q^*$ and similarly $|F|=|cF|$. Thus, we get  the same value if we sum over $F$ non-monic and divide by $q-1$ to account for the leading coefficient. Putting all of this together, we have 
\begin{align}
S_{k,d+1}=& \alpha_k(0)
\sum_{\substack{R_0\in \mathcal{M}_{\leq \left \lfloor \frac{d}{p}\right \rfloor }\\ R_0(0)\not =0}}   \frac{d_k(R_0^p)}{{|R_0|^{\frac{p}{2}}}} \label{sk+1} \\
&+\frac{1}{q-1} \sum_{\substack{\deg(R)\leq \left \lfloor \frac{d}{p}\right \rfloor \\R(0)\not = 0}} \sum_{d < n \leq kd}\frac{\alpha_k(n)}{q^{\frac{n}{2}}}\sum_{\substack{\deg(F)=n\\F(x) \equiv R(x^p) \pmod{x^{d+1}}}}d_k(F) \nonumber \\
=&\alpha_k(0)\sum_{\substack{R_0\in \mathcal{M}_{\leq \left \lfloor \frac{d}{p}\right \rfloor}\\ R_0(0)\not =0}}
\frac{d_k(R_0^p)}{{|R_0|^{\frac{p}{2}}}}+q^{\left \lfloor \frac{d}{p}\right \rfloor}\sum_{d< n \leq kd}\frac{\alpha_k(n)}{q^{\frac{n}{2}}} \sum_{j=0}^{n-d-1} \binom{n-j+k-1}{k-1}\binom{k}{j} (-1)^j q^{n-d-j} \nonumber 
\\
&+O\Bigg( \frac{q}{(1-q^{-1/2})^k}\sum_{\substack{R_0\in \mathcal{M}_{\leq \left \lfloor \frac{d}{p}\right \rfloor}\\ R_0(0)\not =0}} d_k(R_0^p) \sum_{d<n\leq kd} \frac{1}{q^{\frac{n}{2}}} \binom{k-1}{n-d-1} \binom{d+k-1}{k-1} \frac{d+k}{n-d}\nonumber \\
& +\frac{q}{(1-q^{-1/2})^k}\sum_{d<n \leq kd} (k+2)^{n+d} q^{\frac{n}{2p}-\frac{d}{2}+\frac{d}{2p}+1}(n+k-1)^{k-1}\Bigg), \nonumber
\end{align}
where we have used the bound $|\alpha_k(n)|\leq \frac{q}{(1-q^{-1/2})^k}$.

Now we bound the first error term in the equation above. We have 
\begin{align}
 \sum_{\substack{R_0\in \mathcal{M}_{\leq \left \lfloor \frac{d}{p}\right \rfloor}\\ R_0(0)\not =0}} & d_k(R_0^p) \sum_{d<n\leq kd} \frac{1}{q^{\frac{n}{2}}} \binom{k-1}{n-d-1} \binom{d+k-1}{k-1} \frac{d+k}{n-d} \nonumber  \\
& \leq (d+k)^k \sum_{\substack{R_0\in \mathcal{M}_{\leq \left \lfloor \frac{d}{p}\right \rfloor}}} d_k(R_0^p) \sum_{d<n\leq kd} \frac{k^{n-d-1}}{q^{\frac{n}{2}}}  \nonumber \\
& \leq (d+k)^k k^{-1} q^{-d/2} \sum_{m=1}^{d(k-1)} \frac{k^m}{q^{m/2}}\sum_{\substack{R_0\in \mathcal{M}_{\leq \left \lfloor \frac{d}{p}\right \rfloor}}} d_k(R_0^p). \label{summ}
\end{align}
To bound the sum over $R_0$ above, we consider the generating series and we have
\begin{align*}
& \sum_{R_0 \in \mathcal{M}} d_k(R_0^p) u^{\deg(R_0)} = \prod_Q \Big( 1 + \binom{p+k-1}{k-1} u^{\deg(Q)} + \sum_{j \geq 2} \binom{jp+k-1}{k-1} u^{j\deg(Q)}\Big) \\
&= \mathcal{Z}(u)^{\binom{p+k-1}{k-1}} \mathcal{F}(u),
\end{align*}
where $\mathcal{F}(u)$ is given by an Euler product which converges absolutely for $|u|<q^{-1/2}$. Using Perron's formula,  we get that
\begin{align*}
\sum_{\substack{R_0\in \mathcal{M}_{\leq \left \lfloor \frac{d}{p}\right \rfloor}}} d_k(R_0^p)= \frac{1}{2 \pi i} \oint_{|u|=q^{-1-\epsilon}} \frac{\mathcal{F}(u)}{(1-u)(1-qu)^{\binom{p+k-1}{k-1}} u^{\lfloor d/p\rfloor}}\, \frac{du}{u},
\end{align*}

 We shift the contour of integration to $|u| = q^{-1/2-\epsilon}$, and we encounter the pole at $u=1/q$. It then follows that 
 \begin{align}
\sum_{\substack{R_0\in \mathcal{M}_{\leq \left \lfloor \frac{d}{p}\right \rfloor}}} d_k(R_0^p) \ll q^{\frac{d}{p}} \left(\frac{d}{p}\right)^{\binom{p+k-1}{k-1}}. \label{sumr0}
\end{align}

For the sum over $m$ in \eqref{summ},  since $k<q^{1/2}$, we have
$$\sum_{m=1}^{d(k-1)} \frac{k^m}{q^{m/2}} \leq \frac{k}{q^{1/2}-k}.$$
Combining the equation above and \eqref{sumr0} and \eqref{summ}, it follows that 
\begin{align}
 \sum_{\substack{R_0\in \mathcal{M}_{\leq \left \lfloor \frac{d}{p}\right \rfloor}\\ R_0(0)\not =0}} & d_k(R_0^p) \sum_{d<n\leq kd} \frac{1}{q^{\frac{n}{2}}} \binom{k-1}{n-d-1} \binom{d+k-1}{k-1} \frac{d+k}{n-d} \ll (d+k)^k q^{\frac{d}{p}-\frac{d}{2}} d^{\binom{p+k-1}{k-1}},\label{secondet}
\end{align}
 and note that the implied constant above does depend on $q$ and $k$.
 
 Finally, we bound the second error term in \eqref{sk+1} and we have 
\begin{align*}
\ll \frac{q^{\frac{d}{2p}-\frac{d}{2}+2}}{(1-q^{-1/2})^k} & \sum_{d\leq n \leq kd} (k+2)^{n+d} q^{\frac{n}{2p}} (n+k-1)^{k-1}\ll \frac{q^{\frac{d}{2}( \frac{k+1}{p}-1)+2}}{(1-q^{-1/2})^k} (d+1)^k (k+2)^{d(k+1)} k^{k-1},
\end{align*}
and the implicit constant in the error term above is independent of $q$ and $k$. 
Note that this last error term dominates \eqref{secondet} for $d$ large enough.
Putting these together, it follows that
\begin{align*}S_{k,d+1} =& \alpha_k(0)\sum_{\substack{R_0\in \mathcal{M}_{\leq \left \lfloor \frac{d}{p}\right \rfloor}\\ R_0(0)\not =0}}
\frac{d_k(R_0^p)}{{|R_0|^{\frac{p}{2}}}}+q^{\left \lfloor \frac{d}{p}\right \rfloor}\sum_{d< n \leq kd}\frac{\alpha_k(n)}{q^{\frac{n}{2}}} \sum_{j=0}^{n-d-1} \binom{n-j+k-1}{k-1}\binom{k}{j} (-1)^j q^{n-d-j} \\
&+ O \Bigg( \frac{q^{\frac{d}{2}( \frac{k+1}{p}-1)+2}}{(1-q^{-1/2})^k}  (d+1)^k (k+2)^{d(k+1)}k^{k-1} \Bigg).
\end{align*}
\kommentar{\acom{
\begin{align*}
\mathcal{F}(u) &=  \prod_Q \Big( 1 + \binom{p+k-1}{k-1} u^{\deg(Q)} + \sum_{j \geq 2} \binom{jp+k-1}{k-1} u^{j\deg(Q)}\Big) \Big(1 - u^{\deg(Q)} \Big)^{\binom{p+k-1}{k-1}} \\
&= \prod_Q \Big( 1+ \Big(  \binom{2p+k-1}{k-1} - \binom{p+k-1}{k-1}^2+\binom{ \binom{p+k-1}{k-1}}{2}\Big) u ^{2\deg(Q)}+ \ldots\Big)
\end{align*}
 We then get that

}}

Similarly, 
\begin{align*}
S_{k,d}
=&\alpha_k(0)\sum_{\substack{R_0\in \mathcal{M}_{\leq \left \lfloor \frac{d}{p}\right \rfloor}\\ R_0(0)\not =0}}
\frac{d_k(R_0^p)}{{|R_0|^{\frac{p}{2}}}}+q^{\left \lfloor \frac{d}{p}\right \rfloor}\sum_{d\leq  n \leq kd}\frac{\alpha_k(n)}{q^{\frac{n}{2}}} \sum_{j=0}^{n-d} \binom{n-j+k-1}{k-1}\binom{k}{j} (-1)^j q^{n-d-j+1}
\\
&+O \Bigg(\frac{q^{\frac{d}{2}( \frac{k+1}{p}-1)+2}}{(1-q^{-1/2})^k}  (d+1)^k (k+2)^{d(k+1)} k^{k-1}\Bigg).
\end{align*}

Finally, 
 \begin{align}
 \frac{S_{k,d+1}}{q-1} -\frac{S_{k,d}}{q(q-1)}=&\frac{\alpha_k(0)}{q}\sum_{\substack{R_0\in \mathcal{M}_{\leq \left \lfloor \frac{d}{p}\right \rfloor}\\ R_0(0)\not =0}}
\frac{d_k(R_0^p)}{{|R_0|^{\frac{p}{2}}}}-\frac{q^{\left \lfloor \frac{d}{p}\right \rfloor }}{q-1} \sum_{d \leq n \leq kd} \frac{\alpha_k(n)}{q^{\frac{n}{2}}} \binom{d+k-1}{k-1} \binom{k}{n-d} (-1)^{n-d} \nonumber \\
& +O \Bigg(\frac{q^{\frac{d}{2}( \frac{k+1}{p}-1)+2}}{(1-q^{-1/2})^k}  (d+1)^k (k+2)^{d(k+1)}k^{k-1} \Bigg).\label{together}
\end{align}
We remark that the second term above is bounded by
\[
\ll q^{ \frac{d}{p}} \binom{d+k-1}{k-1}  \sum_{d \leq n \leq kd} \frac{1}{q^{\frac{n}{2}}}\binom{k}{n-d} \le q^{\frac{d}{p}} (d+k-1)^{k-1} k^{-d} \sum_{d \leq n \leq d+k} \frac{k^n}{q^{\frac{n}{2}}} \ll q^{\frac{d}{p}-\frac{d}{2}} (d+k-1)^{k-1},\]
where we used the fact that $k <q^{1/2}$. Note that this term is also dominated by the error term in \eqref{together}.

Hence
 we get that
\begin{align*}\frac{1}{|\mathcal{AS}_d^0|}\sum_{f\in \mathcal{AS}_d^0}\mathcal{L}\left(\frac{1}{\sqrt{q}},f,\psi\right)^k = & \frac{\alpha_k(0)}{q}\sum_{\substack{R\in \mathcal{M}_{\leq [\frac{d}{p}]}\\ R(0)\not =0}}
\frac{d_k(R^p)}{{|R|^{\frac{p}{2}}}} + O \Bigg(\frac{q^{ \frac{d}{2} \left( \frac{k+1}{p}-1\right)+2}}{(1-q^{-1/2})^k}  (d+1)^{k} (k+2)^{(k+1)d}k^{k-1}\Bigg).
\end{align*}
In the above, we extend the sum over $R$ to all monic $R$ with $R(0) \neq 0$ at the expense of an error term of size $q^{d(\frac{1}{p}-\frac{1}{2})} d^{\binom{p+k-1}{k-1}}$.  Note that this error term is dominated by the error term in the equation above. Hence we get that
\begin{align*}\frac{1}{|\mathcal{AS}_d^0|}\sum_{f\in \mathcal{AS}_d^0}\mathcal{L}\left(\frac{1}{\sqrt{q}},f,\psi\right)^k = & \frac{\alpha_k(0)}{q}\sum_{\substack{R\in \mathcal{M}\\ R(0)\not =0}}
\frac{d_k(R^p)}{{|R|^{\frac{p}{2}}}} + O \Bigg(\frac{q^{ \frac{d}{2} \left( \frac{k+1}{p}-1\right)+2}}{(1-q^{-1/2})^k}  (d+1)^{k} (k+2)^{(k+1)d}k^{k-1}\Bigg).
\end{align*}

Using an additive character sum to detect $p$th powers, we further write the main term above as
\begin{align*}\label{eq:refereewantsthis}
 \frac{\alpha_k(0)}{q}\sum_{\substack{R\in \mathcal{M}\\ R(0)\not =0}}
\frac{d_k(R^p)}{{|R|^{\frac{p}{2}}}} =  & \frac{1}{p} \sum_{\ell=0}^{p-1}\frac{1}{ \left( 1 -\frac{\xi_p^\ell}{\sqrt{q}} \right)^{k}}\prod_{P\not= x} \Bigg(\frac{1}{p}\sum_{\ell=0}^{p-1}\frac{1}{\left(1-\frac{\xi_p^\ell}{\sqrt{|P|}}\right)^k} \Bigg)\\
= & \prod_{P} \Bigg(\frac{1}{p}\sum_{\ell=0}^{p-1}\frac{1}{\left(1-\frac{\xi_p^\ell}{\sqrt{|P|}}\right)^k} \Bigg).
\end{align*}
Combining the two equations above finishes the proof of Theorem \ref{polyn_k} in the case $k>1$.

\end{proof}

\section{Proof of Theorem \ref{thm_absvalue}}
\label{section:proof2}
We first need to prove the following approximate functional equation.

\begin{lem}[Approximate Functional Equation]
\label{afe1}
For $f \in \as$ and $k \in \mathbb{N}$, we have
 \begin{align*}
 \left|\mathcal{L}\left(\frac{1}{\sqrt{q}},f,\psi\right) \right|^{2k} & =\sum_{\substack{F, H \in \mathcal{M}\\\deg(FH) \leq k(d-1)}} \frac{d_k(F)d_k(H) \psi_f(F) \overline{\psi_f}(H)}{\sqrt{|FH|}} \\
 &+ \sum_{\substack{F, H \in \mathcal{M}\\\deg(FH) \leq k(d-1)-1}} \frac{d_k(F)d_k(H) \psi_f(F) \overline{\psi_f}(H)}{\sqrt{|FH|}}.
 \end{align*}
\end{lem}
\begin{proof}
Using \eqref{fe}, we have
$$ \Big|\mathcal{L}(u,f,\psi) \Big|^{2k}= (qu^2)^{k(d-1)} \Big| \mathcal{L} \Big( \frac{1}{qu},f,\overline{\psi} \Big) \Big|^{2k}.$$
Now
$$\Big|\mathcal{L}(u,f,\psi) \Big|^{2k} = \sum_{n=0}^{2k(d-1)} u^n \sum_{\substack{F, H \in \mathcal{M}\\\deg(FH) = n} } d_k(F) d_k(H)\psi_f(F) \overline{\psi_f}(H) =\sum_{n=0}^{2k(d-1)} a(n)u^n.  $$
From the functional equation above, we get that
$$a(n) = q^{n-k(d-1)} \overline{a(2k(d-1)-n)}.$$
Using this and plugging in $u=\frac{1}{\sqrt{q}}$, we finish the proof. 
\end{proof}

The following result allows us to compute averages of $\psi_f(F)$ with $f$ varying over the family 
$\mathcal{AS}_d^0$. Let 
\[\langle \psi_f(F)\rangle_d=\frac{1}{|\mathcal{AS}_d^0|} \sum_{f\in \mathcal{AS}_d^0} \psi_f(F).\]

\begin{lem} \label{lemma} Let $P_1, \dots, P_s$ be distinct monic irreducible polynomials in $\F_q[x]$ such that $\deg(P_1)+\cdots +\deg(P_s)< d$, and $h_1,\dots,h_s$ integers. Then
 \[\langle \psi_f(P_1)^{h_1}\dots \psi_f(P_s)^{h_s} \rangle_d= \left\{ \begin{array}{ll} 1 & \mbox{if $p \mid h_i$ for $1 \leq i \leq s$}, \\
0 & \mbox{otherwise}. \end{array} \right.\]
\end{lem}
\begin{proof}This is a simple case of \cite[Lemma 9.1]{BDFLS12}.
\end{proof}

We are now ready to begin the proof of Theorem \ref{thm_absvalue}. 

\begin{proof}[Proof of Theorem \ref{thm_absvalue}]

Using Lemma \ref{afe1} for $k=1$, we have  
\begin{align*}&\frac{1}{|\mathcal{AS}_d^0|}\sum_{f\in \mathcal{AS}_d^0}\left|\mathcal{L}\left(\frac{1}{\sqrt{q}},f,\psi\right)\right|^2  \\
=&\frac{1}{|\mathcal{AS}_d^0|}\sum_{f\in \mathcal{AS}_d^0}\sum_{\substack{F, H \in \mathcal{M}\\\deg(FH)\leq d-1}} \frac{ \psi_{f}(F) \overline{\psi_{f}}(H)}{\sqrt{|FH|}}+\frac{1}{|\mathcal{AS}_d^0|}\sum_{f\in \mathcal{AS}_d^0}\sum_{\substack{F, H \in \mathcal{M}\\\deg(FH) \leq d-2}} \frac{ \psi_{f}(F) \overline{\psi_{f}}(H)}{\sqrt{|FH|}}\\
=&\sum_{\substack{F, H \in \mathcal{M}\\\deg(FH)\leq d-1}} \frac{ \langle \psi_{f}(F/H)\rangle_d}{\sqrt{|FH|}}+\sum_{\substack{F, H \in \mathcal{M}\\\deg(FH) \leq d-2}} \frac{\langle \psi_{f}(F/H)\rangle_d}{\sqrt{|FH|}}.
\end{align*}
Thus consider the general sum 
\[S(n) := \sum_{\substack{F, H \in \mathcal{M}\\ \deg(FH)\leq n}} \frac{ \langle \psi_{f}(F/H)\rangle_d}{\sqrt{|FH|}}.\]
By Lemma \ref{lemma}, $\langle \psi_f(F/H)\rangle_d$ is trivial unless $F/H$ is a $p^{\text{th}}$-power. Write $R=(F,H)$ and $F=F_1R$, $H=H_1R$ so that $(F_1,H_1)=1$. 
 Thus $F/H=F_1/H_1$ and we must have $F_1=F_0^p$, $H_1=H_0^p$.  We then have to evaluate the term
$$S(n) = \sum_{\substack{R \in \mathcal{M}_{\leq \frac{n}{2}}}}
 \frac{1}{|R|} \sum_{\substack{F_0, H_0\in \mathcal{M}\\
 p \deg(F_0H_0) \leq n-2\deg(R) \\ (F_0,H_0)=1}} \frac{1}{|F_0H_0|^{\frac{p}{2}}}.$$
First we consider the inner sum. Its generating series is given by 
$$\mathcal{F}(u,v) = \sum_{\substack{F_0, H_0 \in \mathcal{M}\\(F_0,H_0)=1}} \frac{ u^{p\deg(F_0)} v^{p\deg(H_0)}}{|F_0H_0|^{\frac{p}{2}}}.$$ Then we have 
$$ \mathcal{F}(u,v) = \prod_Q \Big(  1+ \sum_{j=1}^{\infty} \frac{u^{pj \deg(Q)}}{|Q|^{\frac{pj}{2}}} + \sum_{j=1}^{\infty} \frac{v^{pj \deg(Q)}}{|Q|^{\frac{pj}{2}}}\Big)=\frac{\mathcal{Z} \left(u^p/q^{\frac{p}{2}} \right)   \mathcal{Z} \left(v^p/q^{\frac{p}{2}} \right)}{  \mathcal{Z} \left(u^pv^p/q^{p} \right)}=\frac{1-q^{1-p}u^pv^p}{(1-q^{1-\frac{p}{2}}u^p)(1-q^{1-\frac{p}{2}}v^p)}.$$ 
Now using Perron's formula \eqref{perron} for the sums over $F_0$ and $H_0$ we get that
\begin{align*}
\sum_{\substack{F_0, H_0 \in \mathcal{M}\\ p \deg(F_0H_0) \leq n-2\deg(R) \\ (F_0,H_0)=1}} \frac{1}{|F_0H_0|^{\frac{p}{2}}}= \frac{1}{(2 \pi i)^2} \oint \oint \frac{\mathcal{F}(u,uv)}{(1-u)(1-v)(uv)^{n-2\deg(R)}} \, \frac{du}{u} \, \frac{dv}{v},
\end{align*}
where the integral takes place over small circles around the origin. 

Introducing the sum over $R$ as well and using Perron's formula \eqref{perron}, we get that 
\begin{align*}
 & S(n)= \frac{1}{(2 \pi i)^3} \oint \oint \oint \frac{\mathcal{F}(u,uv)\mathcal{Z} (u^2v^2z^2/q)}{(1-u)(1-v) (1-z)(uvz)^{n} } \, \frac{dz}{z} \, \frac{du}{u} \, \frac{dv}{v} \\ 
=& \frac{1}{(2 \pi i)^3} \oint \oint \oint  \frac{1-q^{1-p}u^{2p}v^p}{(1-q^{1-\frac{p}{2}}u^p)(1-q^{1-\frac{p}{2}}u^pv^p)(1-u^2v^2z^2)(1-u)(1-v)(1-z) (uvz)^{n} } \, \frac{dz}{z} \, \frac{du}{u} \, \frac{dv}{v},
\end{align*}
where the integral takes place over small circles around the origin. 
Since we need to consider 
\begin{align*}
 S(d-1)+S(d-2),
 \end{align*}
 we will sum the integral expressions for $S(n)$ and $S(n-1)$ and later set $n=d-1$.  Thus, we get
 \begin{align*}
& S(n)+S(n-1)=\\
 &  \frac{1}{(2 \pi i)^3} \oint \oint \oint  \frac{1-q^{1-p}u^{2p}v^p}{ (1-q^{1-\frac{p}{2}}u^p)(1-q^{1-\frac{p}{2}}u^pv^p)(1-uvz)(1-u)(1-v)(1-z) (uvz)^{n}} \, \frac{dz}{z} \, \frac{du}{u} \, \frac{dv}{v}.
 \end{align*}
In the integral above, we can choose the contour to be $|u|=|v|=|z|=q^{-\epsilon}$. In the integral over $z$, we shift the contour of integration to  $|z|=\rho$ and $\rho \to \infty$. Then the integral over $z$ is given by the residues at $z=1$ and $z=\frac{1}{uv}$. 

We write
$$S(n)+S(n-1) = A+B,$$ where $A$ corresponds to the residue at $z=1$ and $B$ corresponds to the residue at $z=1/(uv)$. We have that 
\begin{align*}
A= \frac{1}{(2\pi i)^2} \oint \oint  \frac{1-q^{1-p}u^{2p}v^p}{ (1-q^{1-\frac{p}{2}}u^p)(1-q^{1-\frac{p}{2}}u^pv^p)(1-uv)(1-u)(1-v) (uv)^{n}}  \, \frac{du}{u} \, \frac{dv}{v},
\end{align*}
and 
\begin{align*}
B &= 
\frac{1}{(2 \pi i)^2} \oint  \oint  \frac{1-q^{1-p}u^{2p}v^p }{(1-q^{1-\frac{p}{2}}u^p)(1-q^{1-\frac{p}{2}}u^pv^p)(uv-1)(1-u)(1-v) }  \, du \, dv.
\end{align*}
Note that in the integral for $B$, there are no poles of the integrand inside the contour of integration, so $B=0$. Hence we have 
$$S(n)+S(n-1)=A.$$
In the expression for $A$, we shift the contour over $u$ to $|u|=\rho$ and let $\rho \to \infty$. We encounter poles when $u=1, u=\frac{1}{v}$, $u^pv^p=q^{\frac{p}{2}-1}$  and $u^p = q^{\frac{p}{2}-1}$. Then we have poles at $u=1, u=\frac{1}{v} , u = q^{\frac{1}{2}-\frac{1}{p}} \xi_p^j, u = q^{\frac{1}{2}-\frac{1}{p}} \xi_p^jv^{-1}$, for $j=0,\ldots, p-1$. Thus,  we  have that 
$$S(n)+S(n-1) = A_1+A_{v^{-1}}+\sum_{j=0}^{p-1} \Big(A_{ \xi_p^j}+A_{ \xi_p^jv^{-1}} \Big),$$
where $A_{1}, A_{v^{-1}}$ are the negatives of the residues at $u=1, u=\frac{1}{v}$ respectively, and $A_{ \xi_p^j}, A_{ \xi_p^jv^{-1}}$ are the negatives of the residues at $u = \xi_p^jq^{\frac{1}{2}-\frac{1}{p}} , u = \xi_p^jq^{\frac{1}{2}-\frac{1}{p}} v^{-1}$ respectively.
We have that
$$A_1 = \frac{1}{2 \pi i} \oint \frac{1-q^{1-p}v^p}{(1-q^{1-\frac{p}{2}})(1-q^{1-\frac{p}{2}}v^p)(1-v)^2 v^{n} }  \, \frac{dv}{v}.$$
Now we have a double pole at $v=1$ and poles at $v=\xi_p^j q^{\frac{1}{2}-\frac{1}{p}}$ for $j=0,\ldots, p-1$. We write
$$A_1 = A_{1,1}+\sum_{j=0}^{p-1} A_{1,\xi_p^j},$$ where $A_{1,1}$ corresponds to the pole at $v=1$, and $A_{1,\xi_p^j}$ corresponds to the pole at $v = \xi_p^j q^{\frac{1}{2}-\frac{1}{p}}$. We have
\begin{align*}
A_{1,1} 
=&\frac{(1-q^{1-p})(n+1)}{(1-q^{1-\frac{p}{2}})^2}-
\frac{pq^{1-\frac{p}{2}}(1-q^{-\frac{p}{2}})}{(1-q^{1-\frac{p}{2}})^3},
\end{align*}
and
 \begin{align*}
A_{1,\xi_p^j} = & \frac{1-q^{-\frac{p}{2}}}{(1-\xi_p^j q^{\frac{1}{2}-\frac{1}{p}})^2 \xi_p^{nj} q^{n\left(\frac{1}{2}-\frac{1}{p}\right)} (1-q^{1-\frac{p}{2}})\prod_{\ell\not = j} (1-\xi_p^{j-\ell})} \\
=& \frac{1-q^{-\frac{p}{2}}}{(1-\xi_p^j q^{\frac{1}{2}-\frac{1}{p}})^2 \xi_p^{nj} q^{n\left(\frac{1}{2}-\frac{1}{p}\right)} (1-q^{1-\frac{p}{2}})p}.
\end{align*}
The sum over these residues gives 
\begin{align*}
\sum_{j=0}^{p-1} A_{1,\xi_p^j} = &q^{n\left(\frac{1}{p}-\frac{1}{2}\right)} \frac{1-q^{-\frac{p}{2}}}{p (1-q^{1-\frac{p}{2}})}\sum_{j=0}^{p-1} \frac{1}{(1-\xi_p^j q^{\frac{1}{2}-\frac{1}{p}})^2\xi_p^{nj}}\\
=&q^{n\left(\frac{1}{p}-\frac{1}{2}\right)} \frac{1-q^{-\frac{p}{2}}}{p (1-q^{1-\frac{p}{2}})} S_2(n,q^{\frac{1}{2}-\frac{1}{p}}).
 \end{align*}

We now compute $A_{v^{-1}}$, the negative of the residue of the pole in $A$ coming from $u= \frac{1}{v}$.  We have that
\begin{align*}
A_{v^{-1}} =&-\frac{1}{(2 \pi i)}   \oint  \frac{v^p-q^{1-p} }{ (1-q^{1-\frac{p}{2}})(v^p-q^{1-\frac{p}{2}})(1-v)^2} \, dv.
\end{align*}
In the expression above, we make the change of variables $v \mapsto 1/v$. We get that
\begin{align*}
A_{v^{-1}} =& - \frac{1}{(2 \pi i)}   \oint  \frac{1-q^{1-p} v^p}{ (1-q^{1-\frac{p}{2}})(1-q^{1-\frac{p}{2}}v^p) (1-v)^2} \, dv.
\end{align*}
Recall that we are now integrating over $|v|=q^{\epsilon}$. Hence the integral is equal to the residue of the pole at $v=1$. This gives
\begin{align*}
A_{v^{-1}} &= -\frac{pq^{1-\frac{p}{2}} (1-q^{-\frac{p}{2}})}{  (1-q^{1-\frac{p}{2}})^3}.
\end{align*}

We consider the negatives of the residues at $u=\xi_p^j q^{\frac{1}{2}-\frac{1}{p}}$ and at $u =\xi_p^j q^{\frac{1}{2}-\frac{1}{p}}v^{-1}$. For  $u=\xi_p^j q^{\frac{1}{2}-\frac{1}{p}}$ we get 
\begin{align*}
 A_{\xi_p^j}=\frac{1}{2\pi i} \oint  \frac{(1-q^{-1}v^p)v^{-n}}{p(1-\xi_p^j q^{\frac{1}{2}-\frac{1}{p}})(1-\xi_p^j q^{\frac{1}{2}-\frac{1}{p}}v)(1-v^p)(1-v) (\xi_p^j q^{\frac{1}{2}-\frac{1}{p}})^{n}}  \, \frac{dv}{v}.
\end{align*}
In the above, we shift the contour of integration to $|v|=\rho$ and let $\rho \to \infty$. We encounter poles at $v= \xi_p^k$ for $k=0,\ldots, p-1$ (a double pole at $v=1$ and simple poles at $v= \xi_p^k$ and $k=1,\ldots, p-1$). We then have that 
$$A_{\xi_p^j} = \sum_{k=0}^{p-1} A_{\xi_p^j,\xi_p^k},$$
where $A_{\xi_p^j,\xi_p^k}$ corresponds to the residue at $v=\xi_p^k$.
Computing the residue at $v = \xi_p^k$ for $k=1,\ldots, p-1$ and $v=1$ we get
\begin{align*}
A_{\xi_p^j} &= \sum_{k=1}^{p-1}  \frac{(1-q^{-1})q^{n\left(\frac{1}{p}-\frac{1}{2}\right)}}{p^2(1-\xi_p^{j+k} q^{\frac{1}{2}-\frac{1}{p}})(1-\xi_p^j q^{\frac{1}{2}-\frac{1}{p}})(1-\xi_p^k) (\xi_p^{k+j})^{n}} + \frac{q^{n\left(\frac{1}{p}-\frac{1}{2}\right)}}{p(1-\xi_p^j q^{\frac{1}{2}-\frac{1}{p}})(\xi_p^{j})^{n}} \Bigg(\frac{n (q-1
)}{pq (1-\xi_p^j q^{\frac{1}{2}-\frac{1}{p}})} \\
&+ \frac{p+q-1+pq- \xi_p^j  q^{\frac{1}{2}-\frac{1}{p}}(pq+3q+p-3)}{2pq(1-\xi_p^j q^{\frac{1}{2}-\frac{1}{p}})^2 } \Bigg).
\end{align*}
Now we want to sum over $j=0,\ldots, p-1$ and then over $k=1,\ldots, p-1$. Notice that
\begin{align*}
  & \sum_{j=0}^{p-1} \sum_{k=1}^{p-1}  \frac{\xi_p^{-n(k+j)}}{(1-\xi_p^{j+k} q^{\frac{1}{2}-\frac{1}{p}})(1-\xi_p^j q^{\frac{1}{2}-\frac{1}{p}})(1-\xi_p^k) }\\=&\sum_{j=0}^{p-1} \sum_{k=1}^{p-1}  \frac{\xi_p^{-n(k+j)-k}}{(1-\xi_p^j q^{\frac{1}{2}-\frac{1}{p}})^2(1-\xi_p^k) }-\frac{\xi_p^{-n(k+j)-k}}{(1-\xi_p^j q^{\frac{1}{2}-\frac{1}{p}})^2(1-\xi_p^{j+k} q^{\frac{1}{2}-\frac{1}{p}})}\\
=&\sum_{j=0}^{p-1} \frac{\xi_p^{-nj}}{(1-\xi_p^j q^{\frac{1}{2}-\frac{1}{p}})^2}\sum_{k=1}^{p-1}  \frac{\xi_p^{-(n+1)k}}{(1-\xi_p^k) }-\sum_{j=0}^{p-1} \sum_{k=0}^{p-1}\frac{\xi_p^{-(n+1)(k+j)+j}}{(1-\xi_p^j q^{\frac{1}{2}-\frac{1}{p}})^2(1-\xi_p^{j+k} q^{\frac{1}{2}-\frac{1}{p}})}\\
&+\sum_{j=0}^{p-1} \frac{\xi_p^{-nj}}{(1-\xi_p^j q^{\frac{1}{2}-\frac{1}{p}})^3}\\
=&\sum_{j=0}^{p-1} \frac{\xi_p^{-nj}}{(1-\xi_p^j q^{\frac{1}{2}-\frac{1}{p}})^2}\sum_{k=1}^{p-1}  \frac{\xi_p^{-(n+1)k}}{(1-\xi_p^k) }-\sum_{j=0}^{p-1} \frac{\xi_p^j}{(1-\xi_p^j q^{\frac{1}{2}-\frac{1}{p}})^2}\sum_{\ell=0}^{p-1}\frac{\xi_p^{-(n+1)\ell}}{(1-\xi_p^{\ell} q^{\frac{1}{2}-\frac{1}{p}})}\\
&+\sum_{j=0}^{p-1} \frac{\xi_p^{-nj}}{(1-\xi_p^j q^{\frac{1}{2}-\frac{1}{p}})^3}\\
=&S_2(n,q^{\frac{1}{2}-\frac{1}{p}})\sum_{k=1}^{p-1}  \frac{\xi_p^{-(n+1)k}}{(1-\xi_p^k) } -S_2(-1,q^{\frac{1}{2}-\frac{1}{p}})S_1(n+1,q^{\frac{1}{2}-\frac{1}{p}})+S_3(n,q^{\frac{1}{2}-\frac{1}{p}}).
\end{align*}
We also have 
\begin{align*}
&\sum_{j=0}^{p-1} \frac{\xi_p^{-nj}(pq+p+q-1- \xi_p^j  q^{\frac{1}{2}-\frac{1}{p}}(pq+3q+p-3))}{2pq(1-\xi_p^j q^{\frac{1}{2}-\frac{1}{p}})^3 }\\
=&\sum_{j=0}^{p-1} \frac{\xi_p^{-nj}(pq+p+q-1)}{2pq(1-\xi_p^j q^{\frac{1}{2}-\frac{1}{p}})^2 }
-\frac{ \xi_p^{(1-n)j}  q^{\frac{1}{2}-\frac{1}{p}}(q-1)}{pq(1-\xi_p^j q^{\frac{1}{2}-\frac{1}{p}})^3 }.
\end{align*}

Putting the above together, we have 
\begin{align*}
 \sum_{j=0}^{p-1} A_{\xi_p^j}  =&\frac{q^{n\left(\frac{1}{p}-\frac{1}{2} \right)}(q-1)}{p^2q}\Big[S_2(n,q^{\frac{1}{2}-\frac{1}{p}})\left( \frac{p-1}{2}-[n+1]_p\right)-S_2(-1,q^{\frac{1}{2}-\frac{1}{p}})S_1(n+1,q^{\frac{1}{2}-\frac{1}{p}})\\&+S_3(n,q^{\frac{1}{2}-\frac{1}{p}}) \Big]
 +\frac{q^{n\left(\frac{1}{p}-\frac{1}{2}\right)}}{p} \Big[\frac{n (q-1
)}{pq}S_2(n,q^{\frac{1}{2}-\frac{1}{p}})+\frac{pq+p+q-1}{2pq}S_2(n,q^{\frac{1}{2}-\frac{1}{p}})\\&-\frac{ q^{\frac{1}{2}-\frac{1}{p}}(q-1)}{pq}S_3(n-1,q^{\frac{1}{2}-\frac{1}{p}})
\Big]\\
=&\frac{q^{n\left(\frac{1}{p}-\frac{1}{2} \right)}}{p}S_2(n,q^{\frac{1}{2}-\frac{1}{p}})+\frac{q^{n\left(\frac{1}{p}-\frac{1}{2} \right)}(q-1)}{p^2q}S_2(n,q^{\frac{1}{2}-\frac{1}{p}})(n-[n+1]_p)\\&-\frac{q^{n\left(\frac{1}{p}-\frac{1}{2} \right)}(q-1)}{p^2q}S_2(-1,q^{\frac{1}{2}-\frac{1}{p}})S_1(n+1,q^{\frac{1}{2}-\frac{1}{p}})+\frac{q^{n\left(\frac{1}{p}-\frac{1}{2} \right)}(q-1)}{p^2q}S_3(n,q^{\frac{1}{2}-\frac{1}{p}})\\&
 -\frac{q^{(n-1)\left(\frac{1}{p}-\frac{1}{2}\right)}(q-1)}{p^2q}S_3(n-1,q^{\frac{1}{2}-\frac{1}{p}}).
\end{align*}
For $u =\xi_p^j q^{\frac{1}{2}-\frac{1}{p}}v^{-1}$ we get 
\begin{align*}
A_{\xi_p^jv^{-1}}= - \frac{1}{2\pi i} \oint  \frac{ v^p-q^{-1} }{p(1-\xi_p^j q^{\frac{1}{2}-\frac{1}{p}})(v-\xi_p^j q^{\frac{1}{2}-\frac{1}{p}}) (v^p-1)(1-v) (\xi_p^j q^{\frac{1}{2}-\frac{1}{p}})^{n}
} dv.
\end{align*}
Note that there is no pole of the integrand inside the contour of integration with $|v|=q^{-\epsilon}$, so this integral is equal to $0$.

Finally, putting all the non-zero residues together and setting $n+1=d$ gives 
\begin{align*}
&S(d-1)+S(d-2) = A_{1,1}+A_{v^{-1}}+\sum_{j=0}^{p-1} A_{1,\xi_p^j}+\sum_{j=0}^{p-1}\sum_{k=1}^{p-1} A_{\xi_p^j,\xi_p^k}\\ 
=&\frac{d(1-q^{1-p})}{(1-q^{1-\frac{p}{2}})^2} -\frac{2pq^{1-\frac{p}{2}} (1-q^{-\frac{p}{2}})}{  (1-q^{1-\frac{p}{2}})^3}+ \frac{q^{(d-1)\left(\frac{1}{p}-\frac{1}{2}\right)}(1-q^{-\frac{p}{2}})}{p (1-q^{1-\frac{p}{2}})} S_2(d-1,q^{\frac{1}{2}-\frac{1}{p}})\\
&+\frac{q^{(d-1)\left(\frac{1}{p}-\frac{1}{2} \right)}}{p}S_2(d-1,q^{\frac{1}{2}-\frac{1}{p}})+\frac{q^{(d-1)\left(\frac{1}{p}-\frac{1}{2} \right)}(q-1)}{p^2q}S_2(d-1,q^{\frac{1}{2}-\frac{1}{p}})(d-1-[d]_p)\\&-\frac{q^{(d-1)\left(\frac{1}{p}-\frac{1}{2} \right)}(q-1)}{p^2q}S_2(-1,q^{\frac{1}{2}-\frac{1}{p}})S_1(d,q^{\frac{1}{2}-\frac{1}{p}})+\frac{q^{(d-1)\left(\frac{1}{p}-\frac{1}{2} \right)}(q-1)}{p^2q}S_3(d-1,q^{\frac{1}{2}-\frac{1}{p}})\\&
 -\frac{q^{(d-2)\left(\frac{1}{p}-\frac{1}{2}\right)}(q-1)}{p^2q}S_3(d-2,q^{\frac{1}{2}-\frac{1}{p}}).
\end{align*}

\end{proof}

\section{Agreement with the Random Matrix Models}
\label{section:rmt}
Here, we show that the asymptotic formula in Theorem \ref{thm_absvalue} agrees with the conjectured asymptotic formula for the moments in a family with unitary symmetry. We note that considering the moment with absolute value, as in Theorem \ref{thm_absvalue}, is the more standard moment to consider for a family with expected unitary symmetry.

Starting with the observation of Montgomery and Dyson 
that the zeros of $\zeta(s)$ seem to obey the same distribution patterns as the eigenvalues of large random unitary matrices, random matrix models have been given for families of $L$-functions and have been instrumental in the formulation of conjectures in number theory. Associating a random matrix group to each family of $L$-functions as suggested by the work of Katz and Sarnak \cite{KS}, Keating and Snaith \cite{KS2, Keating-Snaith} used random matrix theory computations to conjecture formulas for moments in families of $L$-functions. 

We reproduce here the expected conjectures in the cases of families of $L$-functions with unitary symmetry. 

Let $\mathcal{F}$ denote a family of $L$-functions. For $f \in \mathcal{F}$, let $c(f)$ denote the conductor of the $L$-function associated to $f$, denoted by $L_f(s)$. Let
$$X^{*} = \Big|  \Big\{ f \in \mathcal{F} \,:  \, c(f) \leq X \Big\}\Big|.$$

In the case of a unitary family, the Keating--Snaith conjecture states the following.
\begin{conj}
\label{unitary}
For  a family $\mathcal{F}$ of $L$-functions with \textit{unitary} symmetry and $k$ a positive integer,
$$ \frac{1}{X^{*}} \sum_{\substack{f \in \mathcal{F} \\ c(f) \leq X}}|L_f(1/2)|^{2k}  \sim a(k) g_{\text{U}}(k) (\log X)^{k^2}, $$
where $a(k)$ is an arithmetic factor depending on the specific family considered, and where
$$g_{\text{U}}(k) = \prod_{j=0}^{k-1} \frac{j!}{(j+k)!}.$$
\end{conj}
In particular, in the case of the second moment ($k=1$), we have $g_{\text{U}}(1)=1$. 

Note that Conjecture \ref{unitary} above is stated for a family of $L$-functions over number fields, but a similar Conjecture can be stated in the function field setting. Namely, for a family $\mathcal{F}$ of $L$-functions over function fields with expected unitary symmetry, one would expect
$$ \frac{1}{D^{*}} \sum_{\substack{f \in \mathcal{F} \\ \log_q c(f) = d}} |L_f(1/2)|^{2k} \sim a(k) g_{\text{U}}(k) d^{k^2},$$ where 
$$D^{*} = \Big| \{ f \in \mathcal{F} \,:\, \log_q |c(f)| =d \} \Big|.$$
Note that the leading order term in Theorem \ref{thm_absvalue}  (apart from the arithmetic factor which depends on $q$ and which corresponds to the factor $a(1)$ in Conjecture \ref{unitary}) matches the conjecture above (it is equal to $1$), under the correspondence:
\begin{equation}\label{eq:correspondence}
\mathcal{F} = \mathcal{AS}_d^0, \,  L_f = \mathcal{L}(u,f,\psi), \,  \log_q |c(f)| = d.
\end{equation}
(Technically, the degree of the conductor is $d+1$, but this leads to the same asymptotic.)

It is also possible to extract a conjecture for moments without absolute value from Keating and Snaith's work \cite{Keating-Snaith} by following the computation in Section 2 of \cite{DLN}. This leads to
$$ \frac{1}{X^{*}} \sum_{\substack{f \in \mathcal{F} \\ c(f) \leq X}}L_f(1/2)^{k} \sim a(k,0),$$
where, as before, $a(k,0)$ is an arithmetic factor. Translating as before, this leads to
$$ \frac{1}{D^{*}} \sum_{\substack{f \in \mathcal{F} \\ \log_q c(f) = d}} L_f(1/2)^{k} \sim a(k,0).$$
Once again, we recover the result of Theorem \ref{polyn_k} under the correspondence \eqref{eq:correspondence}.

Hence the polynomial family is expected to have unitary symmetry, as suggested both by Theorems \ref{polyn_k} and \ref{thm_absvalue} and by the local statistics results due to Entin and Pirani \cite{Ent, EntPir}.

\bibliographystyle{amsalpha}
\bibliography{Bibliography}
\end{document}